\documentclass[11pt]{amsart}
\usepackage{amsfonts,amssymb,amsmath,amsthm,here}

\usepackage{color}  
\usepackage{array}

\theoremstyle{plain}
\newtheorem{theo}{Theorem}

\newtheorem{cor}{Corollary}
\newtheorem{lem}{Lemma}
\newtheorem{prop}{Proposition}

\theoremstyle{definition}

\newtheorem{rem}{Remark}

\makeatletter

\@addtoreset{equation}{section}
\makeatother

\makeatletter
\@namedef{subjclassname@2020}{
  \textup{2020} Mathematics Subject Classification}
\makeatother

\def \Z{\mathbb Z}
\def \Q{\mathbb Q}

\def \C{\mathbb C}
\def \calO{\mathcal O}
\def \calA{\mathcal A}

\def \calD{\mathcal D}

\def \frakX{\mathfrak X}

\def \Ker{\mathop\mathrm{Ker}}
\def \Gal{\mathrm{Gal}}
\def \Tr{\mathrm{Tr}}

\allowdisplaybreaks[4]

\author{Miho Aoki }
\address{Department of Mathematics,
Interdisciplinary Faculty of Science and Engineering,
Shimane University,
Matsue, Shimane, 690-8504, Japan}
\email{aoki@riko.shimane-u.ac.jp}

\subjclass[2020]{Primary 11R04, 11R16,  Secondary 11C08, 11R80. }
\keywords{the ring of integers, simplest cubic fields, Shanks' cubic polynomial}
\thanks{This work was supported by JSPS KAKENHI Grant Number JP21K03181} 

\title[Cyclic cubic fields]{Galois module structure of algebraic integers \\  of cyclic cubic fields}

\begin{document}

\begin{abstract} 
We determine the Galois module structure of the ring of integers for all cyclic 
cubic fields using roots of the generic cyclic cubic polynomial $f_n(X)=X^3-nX^2-(n+3)X-1$. Let $L_n=\Q(\rho_n)$ be a cyclic cubic field with Galois group $G:=
\Gal (L_n/\Q)$, where $\rho_n$ is a root of $f_n (X)$, and $\calO_{L_n}$ the
ring
of integers of $L_n$. We  explicitly give the generator of the free module $\calO_{L_n}$
of rank $1$ over the associated order $\calA_{L_n/\Q}:= \{ x\in \Q [G] \, |\,
x\, \calO_{L_n} \subset \calO_{L_n} \}$ by using the roots of $f_n(X)$.
\end{abstract}
\maketitle
\section{Introduction}\label{sec:intro}

Let $L$ be a number field with Galois group $G=\mathrm{Gal}(L/\Q)$. The ring of integers $\calO_L$ of $L$ is a free $\Z$-module
of rank $[L:\Q]$ whose basis is called {\it an integral basis}. It is a simple but important problem of algebraic number theory
 to find an integral basis for a number field.
On the other hand, the ring of integers has a $\Z [G]$-module structure since the Galois group acts on it.
Furthermore, we  can  consider the module structure over {\it the associated order} $\calA_{L/\Q} := \{ x\in \Q [G]\, | \, x \calO_L \subset \calO_L \}$
which contains $\Z [G]$ as a subring. Finding a minimal generating set as modules over these rings is also important problem.
For an abelian number filed $L$, a conclusive result has been obtained by Leopoldt \cite{Leo} (see also \cite{Le}).  It  was shown that the ring of integers is generated by
Gaussian periods over $\Z [G]$ and a free module of rank $1$ over $\calA_{L/\Q}$ generated by a sum of Gaussian periods. For
example, in the case of quadratic fields $L$ with discriminant $D_L$, the ring of integers has the following Galois module structure:
\[
\calO_L=\begin{cases}
\Z [G] \, \eta_{\mathfrak f_L} =\calA_{L/\Q} \, \eta_{\mathfrak f_L} & \text{if $2\nmid D_L$\, (tamely ramified)}, \\
\Z [G]\,  \eta_{\mathfrak f_L}  \oplus \Z =\calA_{L/\Q} (\eta_{\mathfrak f_L}+1) & \text{if $4 | D_L$\, (wildly ramified)},
\end{cases}
\]
where $\mathfrak f_L=|D_L|$ is the conductor of $L$,  and $\eta_{\mathfrak f_L}:= \mathrm{Tr}_{\Q (\zeta_{\mathfrak f_L})/L}
(\zeta_{\mathfrak f_L})$ is {\it a Gaussian period} which is the trace from $\Q (\zeta_{\mathfrak f_L})$ to $L$ 
of a primitive $\mathfrak f_L$th root of unity.
However,  by using a root of the defining polynomial $X^2-D_L$,
the Galois module structure is generally known as follows:
\[
\calO_L=\begin{cases}
\Z [G] \, \frac{1+\sqrt{D_L}}{2} =\calA_{L/\Q} \, \frac{1+\sqrt{D_L}}{2} & \text{if $2\nmid D_L$}, \\
&  \\
\Z [G]\,   \frac{\sqrt{D_L}}{2}  \oplus \Z =\calA_{L/\Q} (\frac{\sqrt{D_L}}{2}+1) & \text{if $4 | D_L$}.
\end{cases}
\]

In this paper, we give the Galois module structure as described above  for cyclic cubic fields using the roots of the generic cyclic cubic polynomial.
Let $n$ be a rational number.
We consider the generic cyclic cubic polynomial defined by
\begin{equation}\label{eq:fn}
f_n(X)=X^3-nX^2-(n+3)X-1
\end{equation}
(see \cite[chap.~1]{Se}).
Assume that $f_n(X)$ is irreducible over $\Q$.
For a root $\rho_n$ of $f_n(X)$, let
$L_n=\Q(\rho_n)$ be the cyclic cubic field.
If $n\in \Z$, then $f_n (X)$ is always irreducible, and the field $L_n$ is known as 
{\it the  simplest cubic field} (see \cite{S} for details).
Let $\sigma$ be the generator of $\Gal(L_n/\Q)$ that satisfies
$\sigma (\rho_n)=-1/(1+\rho_n)$.
For any $x\in L_n$, we put 
$x':=\sigma (x)$ and $x'':=\sigma^2 (x)$.
We have  $L_n=L_{-n-3}$ as $f_n(X)=-X^3f_{-n-3}(1/X)$.
We express $n=n_1/n_2$ where the integers $n_1$ and $n_2$ are coprime (the choice of
$(n_1,n_2)$ can be either $(n_1,n_2)$ or $(-n_1,-n_2)$). Put $\Delta_n=n_1^2+3n_1n_2 +9n_2^2
=de^2c^3=A_nA_n'$, where $d,e,c \in \Z_{>0}$ with $d$ and $e$
being square-free and $(d,e)=1$, and $A_n=n_1+3n_2(1+\zeta_3),\, A_n'=n_1+3n_2 (1+\zeta_3^2)$.
The discriminant of $f_n(X)$ is $d(f_n)=\Delta_n^2/n_2^4=(n^2+3n+9)^2 $.

If $3\nmid n_1$, or $n_1=3t\, (t\in \Z)$ and $n_2 \equiv -2t \pmod{9}$, then $L_n/\Q$ is tamely ramified (cf. Lemma~\ref{lem:m} and (\ref{eq:f})).
In this case, the Galois module structure of the ring of integers is given by the following theorem:
\begin{theo}[=Theorem~\ref{theo:tame}]\label{theo:main1}
Let $n=n_1/n_2$ be a rational number where the integers $n_1$ and $n_2$ are coprime.
Suppose that the cubic polynomial $f_n(X)$ is irreducible over $\Q$, and $3\nmid n_1$ or
$n_1=3t\, (t\in \Z)$, $n_2 \equiv -2t \pmod{9}$. 
There exist integers $a_0$ and $a_1$  that satisfy $ec=a_0^2-a_0a_1+a_1^2$ and
$a_0+a_1\zeta_3 $ divides $ A_n$ in $\Z [\zeta_3]$.
Furthermore, for any such $a_0$ and $a_1$, we  have $(\varepsilon ec^2-n_1(a_0+a_1))/3\in \Z$ for $\varepsilon \in \{ \pm 1\} $ given by
\[
\varepsilon: =
\begin{cases}
\left( \frac{n_1 (a_0+a_1)}{3} \right) & \text{if $3\nmid n_1$}, \\
\left( \frac{n_2a_0}{3} \right) & \text{if $n_1=3t\, (t\in \Z),\, n_2\equiv -2t \!\! \pmod{9}$}, 
\end{cases}
\]
where $\left( \frac{\cdot}{3} \right)$ is the Legendre symbol, and 
\begin{equation*}
\alpha :=\frac{1}{ec^2}\left( n_2(a_0\rho_n+a_1\rho_n')+\frac{1}{3} (\varepsilon ec^2-n_1(a_0+a_1) ) \right)
\end{equation*}
is a generator of a normal integral basis of the cyclic cubic field $L_n$, namely we have
$\calO_{L_n}=\Z [G]\, \alpha =\calA_{L_n/\Q} \, \alpha$.
\end{theo}

If $n_1=3t\, (t\in \Z)$ and $n_2 \not\equiv -2t \pmod{9}$, then $L_n/\Q$ is wildly ramified (cf. Lemma~\ref{lem:m} and (\ref{eq:f})).
In this case, the Galois module structure of the ring of integers is given by the following theorem.
\begin{theo}[=Theorem~\ref{theo:wild}]\label{theo:main2}
Let $n=n_1/n_2$ be a rational number where the integers $n_1$ and $n_2$ are coprime.
Suppose that the cubic polynomial $f_n(X)$ is irreducible over $\Q$, and $n_1=3t\, (t\in \Z)$, $n_2 \not\equiv -2t \pmod{9}$. There exist integers $a_0$ and $a_1 $  that satisfy
\[ec=
\begin{cases}
a_0^2-a_0a_1+a_1^2 &  \text{ if $t\equiv n_2 \!\! \pmod{3}$}, \\
3(a_0^2-a_0a_1+a_1^2)  &   \text{ if $t\not\equiv n_2 \!\! \pmod{3}$},
\end{cases}
\]
and
$a_0+a_1 \zeta_3 $ divides $A_n$ in $\Z [\zeta_3]$. Furthermore, for any such $a_0,\, a_1$ and 
\[
\alpha :=\frac{1}{ec^2}(3n_2 (a_0\rho_n +a_1 \rho'_n)-n_1(a_0+a_1)),
\]
we have $\alpha \in {\bf e}_{\mathfrak f_{L_n}}\calO_{L_n}$ and 
\[
\calO_{L_n}=\Z [G]\, \alpha \oplus \Z =\calA_{L_n/\Q}  (\alpha+1),
\]
where ${\bf e}_{\mathfrak f_{L_n}}=(2-\sigma-\sigma^2)/3$.
\end{theo}

Using the results of this paper, the author \cite{A1,A2} gave a linear relation between a cubic Gaussian period and a root of the generic cyclic cubic polynomial $f_n(X)$ generally,
and derived the minimal polynomial of the cubic Gaussian period which was given by Gauss and Hasse \cite{G,H}.

For the proofs of the Theorems~\ref{theo:main1} and \ref{theo:main2}, we use
the method proposed by Acciaro and Fieker \cite{AF}, which constructs a normal integral basis from a normal 
basis and an integral basis for a tamely ramified case.
Using Leopoldt's theorem, 
we decompose the ring of integers into direct sums as $\Z [G]$-modules and then use Acciaro and Fieker's method 
(including wildly ramified cases) for the direct summands.
In order to use their method, we shall first  explicitly give an integral basis of a cyclic cubic field using roots of  the generic cyclic cubic
polynomial  (Proposition~\ref{prop:IB}) in \S \ref{sec:IB}. We use  a result of Albert \cite{Al} that gives an integral basis of a cubic
field defined by a polynomial $X^3+aX+b$ with $a,b\in \Z$.
In \S~\ref{sec:BA}, we shall recall the result given by Leopoldt, which gives the Galois module structure of the ring of integers 
for abelian number fields. In \S~\ref{sec:tame}, we shall consider the tamely ramified cases,  prove Theorem~\ref{theo:main1} and
give some corollaries (Corollaries~\ref{cor:tame1} and \ref{cor:tame2}) for the special cases. 
In \S~\ref{sec:wild}, we shall consider the wildly ramified cases, prove Theorem~\ref{theo:main2} and give some corollaries 
 (Corollaries~\ref{cor:wild1} and \ref{cor:wild2}) for the special cases. 
In \S~\ref{sec:ex}, we give some numerical examples (Tables~\ref{table:2}-\ref{table:11}) by using Magma \cite{BCP}.

%\begin{theo}[=Corollaries 2-5]\label{theo:cors}
%Let $n=n_1/n_2$ be a rational number where the integers $n_1$ and $n_2$ are coprime, and $\rho_n,\, \rho_n', \, \rho_n''$ roots of $f_n(X)=X^3-nX^2-(n+3)X-1$
%satisfying $\rho_n'=-1/(1+\rho_n)$.
%Put $\Delta_n:=n_1^2+3n_1n_2+9n_2^2$.
%Suppose that the cubic polynomial $f_n(X)$ is irreducible.
%\begin{enumerate}
%\item[$(1)$] If $\Delta_n$ is square-free $($in this case, we have $3\nmid n_1)$, then we have $\calO_{L_n}=\Z [G] \cdot \alpha \, (=\calA_{L_n/\Q} \cdot 
%\alpha )$ for $\alpha :=n_2 \rho_n +\left( \left(\frac{n_1}{3} \right)-n_1 \right)/3$.
%\item[$(2)$] If $\Delta_n=3^3d$ with a square-free integer $d$ and $n_1=3t\, (t\in \Z),\, n_2 \equiv -2t \pmod{9}$, 
% then we have $\calO_{L_n}=\Z [G] \cdot \alpha \, (=\calA_{L_n/\Q} \cdot 
%\alpha )$ for $\alpha :=\left(n_2( \rho_n -\rho_n')+3\left( \left(\frac{n_1}{3} \right) \right) \right)/9$.
%\item[$(3)$]  If $\Delta_n=3^3d$ with a square-free integer $d$ and $n_1=3t\, (t\in \Z),\, n_2 \not\equiv -2t \pmod{9}$, 
% then we have $\calO_{L_n}=\Z [G] \cdot \alpha =\calA_{L_n/\Q} \cdot 
%(\alpha +1)$ for $\alpha :=n_2( \rho_n -\rho_n')/3$.
%\item[$(4)$] If $\Delta_n=3^2d$ with a square-free integer $d$ and $3\nmid d$, 
% then we have $\calO_{L_n}=\Z [G] \cdot \alpha =\calA_{L_n/\Q} \cdot (\alpha +1)
%$ for $\alpha :=( 3n_2 \rho_n -n_1)/3$.
%\end{enumerate}
%\end{theo}

%
%\section{Generic cyclic cubic polynomial}
%
%
\section{Integral Basis}\label{sec:IB}

 In this section, we give an integral basis of $L_n$. We use the results of Albert (\cite{Al}) that
explicitly give an integral basis of a cubic field.
Consider a cubic polynomial
\begin{equation}\label{eq:f(X)}
f(X)=X^3+aX+b \quad (a,b\in \Z),
\end{equation}
that is irreducible over $\Q$. We assume that there is no prime number $p$ such that $p^2 |a$ and $p^3 |b$.
The discriminant of $f(X)$ is $d(f)=-\Delta$ with $\Delta :=4a^3+27b^2$. 
Let $v_p(x)$ denote the normalized $p$-adic valuation of $x\in \Q$ for a prime number $p$.
Let $P$ be the largest integer divisor of $a$ such that $(P,6)=1$ and $P^2 |b$. From our
assumption on $a$ and $b$, the integer $P$ must be  a product of distinct primes.
Put $Q= \prod_{(p,6a)=1} p^{ \lfloor v_p(\Delta)/2 \rfloor }$, where the product runs over all 
prime numbers $p$ satisfying $(p,6a)=1$.
Then, we can write 
\begin{equation}\label{eq:Delta}
\Delta =4a^3+27b^2=2^{2\lambda} 3^{2\mu} P^2 Q^2 \Delta',
\end{equation}
where $\lambda, \mu \in Z_{\geq 0}$, and $\Delta'$ is the largest integer divisor of $\Delta /(PQ)^2$ satisfying
$4 \nmid \Delta'$ and $9\nmid \Delta'$ (\cite[Theorem~2]{Al}). Let $\theta$ be a root of $f(X)$.
Albert gave an integral basis of $\Q(\theta)$ for both cases where $a\equiv 6 \pmod{9},\, b\not\equiv 0 \pmod{3}$, and $\mu >2$ hold (
\cite[Theorem~3]{Al}), and where it does not hold (\cite[Theorem~4]{Al}). Let $\psi (x)=\varphi (x)-1$ for $x\in \Z_{>0}$, where 
$\varphi (x)$ is the Euler's totient function.

First, assume that $a\equiv 6\pmod{9},\, b\not\equiv 0\pmod{3}$, and $\mu >2$ hold, then an integral basis of $\Q (\theta )$
is given by
\begin{equation}\label{eq:Albert-IB1}
\left\{ 
1, \frac{-r+\theta}{3} ,\ \frac{r^2+a+r\theta+\theta^2}{2^{\varepsilon_2}3^{\mu-1} PQ}
\right\},
\end{equation}
where the integers $\varepsilon_2 $ and $r$ are given as follows. Put $a=3a_1$ with $a_1 \in \Z$.

If $b=2b_1 (b_1 \in \Z),\, b_1 \equiv 1 \pmod{2},\, a\equiv 1 \pmod{4}$, and $\Delta' \equiv 3 \pmod{4}$,
then we have
\begin{align}
\varepsilon_2 & =\lambda,  \label{eq:Albert-1-1}\\
r & =-2b_1 [ (2^{\lambda +1} 3^{\mu} a_1)^{\psi (Q)} 3^{\mu} +(2^{ \lambda +1 } Qa_1)^{\psi (3^{\mu-1})}Q ]2^{\lambda}
+(2^{\lambda -1}-3b_1)(a_1 Q 3^{\mu})^{\psi (2^{\lambda}) } \cdot Q\cdot 3^{\mu-1} \notag
\end{align}

If $b=2b_1 (b_1 \in \Z),\, b_1 \equiv 1 \pmod{2},\, a\equiv 1 \pmod{4}$, and $\Delta' \not\equiv 3 \pmod{4}$,
then we have
\begin{align}
\varepsilon_2 & =\lambda -1, \notag \\
r & =-2^{\lambda} b_1 [ 3^{\mu} (2^{\lambda}  3^{\mu}  a_1)^{\psi (Q)} + Q(2^{\lambda} Qa_1)^{\psi (3^{\mu-1})}]
-3^{\mu}  b_1 Q (3^{\mu} Qa_1)^{\psi (2^{\lambda-1})}. \label{eq:Albert-1-2}
\end{align}

If $b\equiv 0 \pmod{4}$ and $a\not\equiv 1 \pmod{4}$, 
then we have
\begin{align}
\varepsilon_2 & =1, \notag \\
r & =-b [ (2 a_1  3^{\mu}  )^{\psi (Q)} 3^{\mu} + Q(2a_1Q)^{\psi (3^{\mu-1})}]
+3^{\mu}  a_1 Q. \label{eq:Albert-1-3}
\end{align}

For all other cases, we have
\begin{align} 
\varepsilon_2 & =0, \notag \\
r & =-b [ (2 a_1  3^{\mu}  )^{\psi (Q)} 3^{\mu} + Q(2a_1Q)^{\psi (3^{\mu-1})}]. \label{eq:Albert-1-4}
\end{align}

Next, assume  that the conditions $a\equiv 6 \pmod{9},\, b\not\equiv 0 \pmod{3}$, and $\mu >2 $ are not simultaneously
satisfied. Then, an integral basis of $\Q (\theta)$ is given by
\begin{equation}\label{eq:Albert-IB2}
\left\{ 1,\ \theta, \ \frac{r^2+a+r\theta +\theta^2}{3^{\varepsilon_1}  2^{\varepsilon_2} PQ} \right\},
\end{equation}
where 
\begin{equation}\label{eq:Albert-2-1}
\varepsilon_1 =\begin{cases}
1, & \text{ if $b\equiv 0 \pmod{9},\ a\equiv 0 \pmod{3}$ or $b\not\equiv 0 \pmod{3},\, b^2+a-1 \equiv 0 \pmod{9} $}, \\
0, & \text{otherwise},
\end{cases}
\end{equation}
and $\varepsilon_2$ and $r$ are given as follows. In all cases, we define $a^{\psi (Q)} =0$ when $a=0$
(hence $Q=1$ and $\psi (Q)=0$).

If $b=2b_1 (b_1 \in \Z),\, b_1 \equiv 1 \pmod{2},\, a\equiv 1 \pmod{4}$, and $\Delta' \equiv 3 \pmod{4}$,
then we have
\begin{align}
\varepsilon_2 & =\lambda,  \label{eq:Albert-2-2}\\
r & =-3\cdot 2^{\lambda} \cdot b  (2^{\lambda +1} a)^{\psi (Q)}+ 3Q(2^{ \lambda -1 } -3b_1)
(3a Q )^{\psi (2^{\lambda}) }  -2^{2\lambda} b Q^2. \notag
\end{align}

If $b=2b_1 (b_1 \in \Z),\, b_1 \equiv 1 \pmod{2},\, a\equiv 1 \pmod{4}$, and $\Delta' \not\equiv 3 \pmod{4}$,
then we have
\begin{align}
\varepsilon_2 & =\lambda -1, \notag \\
r & =-3b_1 [ 2^{\lambda}  (a2^{\lambda})^{\psi (Q)} + Q(aQ)^{\psi (2^{\lambda-1})}]
-b2^{2\lambda} Q^2. \label{eq:Albert-2-3}
\end{align}

If $b\equiv 0 \pmod{4}$ and $a\not\equiv 1 \pmod{4}$, 
then we have
\begin{align}
\varepsilon_2 & =1, \notag \\
r & =-b [ 3(2a)^{\psi (Q)} +Q^2] +3Q^2 a. \label{eq:Albert-2-4}
\end{align}

For all other cases, we have
\begin{align} 
\varepsilon_2 & =0, \notag \\
r & =-b [ 3(2 a )^{\psi (Q)} +Q^2]. \label{eq:Albert-2-5}
\end{align}

To use these results, we transform the polynomial $f_n(X)$ of (\ref{eq:fn}) into the form (\ref{eq:f(X)}) with no prime number $p$
such that $p^2 |a$ and $p^3 |b$. Assume that $f_n(X)$ is irreducible over $\Q$ (it is always irreducible whenever $v_2(n) \geq 0$;
otherwise, it is not necessarily irreducible. For example, $f_{-3/2} (X)=(X-1)(X+2)(X+1/2)$).
We have
\[
(3n_2)^3 f_n \left( \frac{X}{3n_2} +\frac{n}{3} \right)=X^3-3\Delta_n X -(2n_1+3n_2)\Delta_n.
\]
Let $m$ be the largest positive integer such that $m^2$ divides $3\Delta_n$ and $m^3$ divides $(2n_1+3n_2) \Delta_n$.
Put
\begin{equation}\label{eq:hn}
h_n(X)= \left( \frac{3n_2}{m} \right)^3 f_n \left( \frac{mX}{3n_2}+\frac{n}{3} \right)=X^3+aX+b,
\end{equation}
with integers $a:=-3\Delta_n/m^2$ and $b:= -(2n_1+3n_2) \Delta_n/m^3$. Then, we have $h_n (X) \in \Z[X]$,
and
there is no prime number $p$ such that $p^2 |a$ and $p^3 |b$. The polynomial $h_n (X)$ is irreducible over $\Q$ since $f_n (X)$ is
irreducible by the assumption. Furthermore,
\begin{equation}\label{eq:theta}
\theta := \frac{3n_2}{m} \left( \rho_n-\frac{n}{3} \right)
\end{equation}
is a root of $h_n (X)$ and $L_n=\Q (\rho_n)=\Q (\theta)$. We express $\Delta_n=n_1^2+3n_1n_2+9n_2^2=de^2c^3$ using $d,e,c\in \Z_{>0}$,
where $d$ and $e$ are square-free, and $(d,e)=1$.
Let $\zeta=\zeta_3$ be a primitive cube root of unity. We have $\Delta_n=A_n A_n'$, where $A_n:=n_1+3n_2(1+\zeta)$ and $A_n' := n_1+3n_2(1+\zeta^2)$.
Let $\calO_{L_n}$ be the ring of integers of $L_n$.
\begin{lem}\label{lem:n2rho}
Let $n=n_1/n_2$ be a rational number where $n_1$ and $n_2$ are $($not necessarily coprime$)$ integers.  We have $n_2 \rho_n \in \calO_{L_n}$.
\end{lem}
\begin{proof}
The assertion follows from the fact that $n_2 \rho_n$ is a root of a monic polynomial :
\[
n_2^3 f_n (X/n_2) =
X^3-n_1X^2-(n_1n_2+3n_2^2)X-n_2^3 \quad (\in \Z [X]).
\]
\end{proof}
\begin{lem}\label{lem:1-zeta}
Let $\mathfrak p$ be a prime ideal of $\Z [\zeta]$ containing both $A_n$ and $A_n'$.
Then, we have $\mathfrak p =(1-\zeta)$.
\end{lem}
\begin{proof}
Assume that a prime ideal $\mathfrak p$ contains both $A_n$ and $A_n'$, then the prme ideal $\mathfrak p$ contains $A_n-A_n' =3n_2 \zeta (1-\zeta)$.
Since $\mathfrak p$ is a prime ideal, it follows that $1-\zeta \in \mathfrak p$ or $n_2 \in \mathfrak p$.
If $n_2 \in \mathfrak p$, then we have $n_1 =A_n-3n_2 (1+\zeta )\in \mathfrak p$, and this is a contradiction since $n_1, n_2 \in
\mathfrak p \cap \Z=p\Z$ for a prime number $p$. We conclude that $\mathfrak p =(1-\zeta)$.
\end{proof}
\begin{lem}\label{lem:equiv-0,1}
If $p$ is a prime number satisfying $p |\Delta_n$, then we have $p\equiv 0,1 \pmod{3}$.
\end{lem}
\begin{proof}
Assume that $p$ is a prime number satisfying $p |\Delta_n$ and $p\equiv -1 \pmod{3}$. Since $p$ is 
inert in $\Q (\zeta)$, we have $p |A_n$ and $p |A_n'$. However, since $\{ 1, \zeta \}$ is a basis of
$\Z [\zeta]$ over $\Z$ and $p |A_n =n_1+3n_2(1+\zeta)$, we have $ p|(n_1+3n_2)$ and $p |3n_2$. This is a contradiction
since $p\ne 3$ and $(n_1,n_2)=1$.
\end{proof}
\begin{lem}\label{lem:prime-3}
The only prime number that divides both $\Delta_n $ and $2n_1+3n_2$ is $3$.
\end{lem}
\begin{proof}
Assume that $p$ is a prime number satisfying $p|\Delta_n$ and $p|(2n_1+3n_2)$.
From $p |(2n_1+3n_2)$, we have $2n_1 \equiv -3n_2 \pmod{p}$, and hence $4\Delta_n \equiv 27 n_2^2 \pmod{p}$.
From $p|\Delta_n$, we have $p |27n_2^2$. Assume $p|n_2$. From $2n_1 \equiv -3n_2 \pmod{p}$,  we have $p |n_1$.
Since $p\ne 2$ from Lemma~\ref{lem:equiv-0,1}, we obtain $p|n_1$. However, this is a contradiction since $(n_1,n_2)=1$.
We conclude  $p\nmid n_2$ and hence $p=3$.
\end{proof}

The integer $m$ in (\ref{eq:hn}) is given by the following lemma:
\begin{lem}\label{lem:m}
Let $d,e,c$, and $m$ be the inetegers as defined above, We have the following:
\begin{enumerate}
\item[(1)] If $3\nmid n_1$, then we have $3\nmid d,\, 3 \nmid e,\, 3\nmid c $ and $m=c$.
\item[(2)] If $n_1=3t \, (t\in \Z)$ and $n_2 \equiv -2t \pmod{9}$, then we have $ 3\nmid d,\, 3\nmid e,\, 3 ||c$ and $m=3c$.
\item[(3)] If $n_1=3t\, (t\in \Z)$ and $n_2 \not\equiv -2t \pmod{9}$, then we have the following.
\begin{enumerate}
\item[(i)]  If $t\equiv n_2 \pmod{3}$, then $3\nmid d,\, 3\nmid e,\, 3||c$ and $m=c$.
\item[(ii)]  If $t\not\equiv n_2 \pmod{3}$, then $3 \nmid d,\, 3 ||e,\, 3\nmid c $ and $m=3c$.
\end{enumerate}
\end{enumerate}
\end{lem}
\begin{proof}
The assertions on $d,\, e$, and $c$ follow from $\Delta_n=n_1^2+3n_1n_2 +9n_2^2=de^2c^3$ and $(n_1.n_2)=1$.
We show the assertions on $m$. Next, we prove the assertions on $m$.

(1)\ Suppose that $3\nmid n_1$ holds. Since $m^2 |3\Delta_n$ and $3\nmid \Delta_n$, we have $3\nmid m$.
Therefore, we have $(m, 2n_1+3n_2)=1$ from Lemma~\ref{lem:prime-3}. 
From $m^3 | (2n_1+3n_2)\Delta_n$, we obtain $m^3|\Delta_n =de^2c^3$.
Let $p$ be a prime number satisfying $p|m$. Since $d$ and $e$ are square-free and $(d,e)=1$, we have
\[
3v_p (m)=v_p (m^3) \leq v_p(de^2c^3)=\begin{cases}
1+3v_p (c) & \text{if $p|d$}, \\
2+3v_p (c) & \text{if $p|e$}.
\end{cases}
\]
In both cases $p|d $ and $p|e$, we obtain $v_p (m)\leq v_p (c)$.
Since this inequality holds for any prime number $p$ dividing $m$, we obtain $ m|c$. Let $c'$ be the
integer satisfying $c=mc'$.
Since $3\Delta_n /m^2=3de^2{c'}^3m$ and $(2n_1+3n_2)\Delta_n/m^3=(2n_1+3n_2)de^2{c'}^3$, we have $c'=1$ from the choice of $m$. We conclude 
$m=c$.

(2)\ Suppose that $n_1=3t \,  (t\in \Z)$ and $n_2 \equiv -2t \pmod{9}$ hold.
We have $3^3 || \Delta_n$ and $2t+n_2 \equiv 0 \pmod{9}$. From $m^2 |3\Delta_n$ and $m^3 | (2n_1+3n_2)\Delta_n =3(2t+n_2)\Delta_n$, we obtain 
$3^2 ||m$. Since $m^2 |3\Delta_n$, $3^2||m$, and $2n_1+3n_2=3(2t+n_2) \equiv 0 \pmod{3^3}$, it follows from Lemma~\ref{lem:prime-3} that
$m/9$ and $(2n_1+3n_2)/27$ are coprime. On the other hand, we have $(m/9)^3 |(2n_1+3n_2) \Delta_n/3^6=(2n_1+3n_2)/3^3\times
\Delta_n/3^3$.
We conclude that $(m/9)^3| \Delta_n /3^3=de^2c^3/3^3$.
Let $p$ be a prime number satisfying $p|m$ and $p\ne 3$. Since $d$ and $e$ are square-free and $(d,e)=1$, we have
\[
3v_p (m)=v_p ((m/9)^3) \leq v_p (de^2c^3/3^3)=v_p (de^2c^3)=\begin{cases}
1+3v_p(c) & \text{if $p|d$}, \\
2+3v_p (c) & \text{if $p|e$}.
\end{cases}
\]
In both cases $p|d$ and $p|e$, we obtain $v_p (m )\leq v_p (c)$. Since this inequality holds for any prime number $p \, (\ne 3)$
dividing $m$, we obtain $m/9 |c$.
From $m/9 |c,\, 3\nmid m/9$ and $3||c$, we can write $c=mc'/3$ with $c'\in \Z$. Since $3\Delta_n /m^2 =de^2{c'}^3 m/9$ and
$(2n_1+3n_2)\Delta_n/m^3=(2n_1+3n_2)/27 \times de^s{c'}^3$, we have $c'=1$ from the choice of $m$. We conclude that $m=3c$.

(3)\ Suppose that $n_1=3t\, (t\in \Z)$ and $n_2 \not\equiv -2t \pmod{9}$ hold. First, we prove $3 ||m$.
If $t\equiv n_2 \pmod{3}$, then we have $3^3 ||\Delta_n$ and $2t+n_2 \not\equiv 0\pmod{9}$.
From $m^2 |3\Delta_n$ and $m^3 | (2n_1+3n_2)\Delta_n =3(2t+n_2)\Delta_n$, we obtain $3||m$. If $t\not\equiv n_2 \pmod{3}$, then we have
$3^2 || \Delta_n$. From $m^2 |3\Delta_n$ and $m^3 | (2n_1+3n_2)\Delta_n =3(2t+n_2)\Delta_n$, we obtain $3||m$.
Next, we prove that $m/3$ divides $c$.
Since $m^2 |3\Delta_n$ and $3||m$, it follows from Lemma~\ref{lem:prime-3} that $m/3$ and $(2n_1+3n_2)/3 $ are coprime .
On the other hand, we have $(m/3)^3 | (2n_1+3n_2)\Delta_n /3^3=(2n_1+3n_2)/3 \times \Delta_n/3^2$. We conclude
$(m/3)^3 |\Delta_n/3^2 =de^2c^3/3^2$. Let $p$ be a prime number satisfying $p|m$ and $p\ne 3$.
Since $d$ and $e$ are square-free and $(d,e)=1$, we have
\[
3v_p(m)=v_p ((m/3)^3)\leq v_p (de^2c^3/3^2)=v_p(de^2c^3)=\begin{cases}
1+3v_p(c) & \text{if $p|d$},\\
2+3v_p (c) & \text{if $p|e  $}.
\end{cases}
\]
In both cases $p|d$ and $p|e$, we obtain $v_p (m)\leq v_p(c)$.
Since this inequality holds for any prime number $p\, (\ne 3)$ dividing $m$, we obtain $m/3 |c$. Finally, we show $m=c$ if $t\equiv n_2
\pmod{3}$ and $m=3c$ if $t\not\equiv n_2 \pmod{3}$. If $t\equiv n_2 \pmod{3}$, then we can write $c=mc'$ with $c' \in \Z$ from
$m/3 |c,\, 3\nmid m/3$ and $3 ||c$. Since $3\Delta_n /m^2=3de^2 {c'}^3m$ and $(2n_1+3n_2)\Delta_n /m^3=(2n_1+3n_2)de^2{c'}^3 $,
we have $c'=1$ from the choice of $m$. We conclude  that $m=c$. If $t\not\equiv n_2 \pmod{3}$, then we can write $c=mc'/3$ with $c' \in \Z$ from
$m/3 |c$. Since $3\Delta_n /m^2=d(e/3)^2 {c'}^3m$ 
and $(2n_1+3n_2)\Delta_n /m^3=(2n_1+3n_2)/3 \times d(e/3)^2 {c'}^3$, we have $c'=1$ from the choice of $m$. We conclude that $m=3c$.
\end{proof}

Using the results of Albert, we give an integral basis of $L_n =\Q (\rho_n)$.
\begin{lem}\label{lem:lambda-mu}
Put $a=-3\Delta_n/m^2$ and $b=-(2n_1+3n_2)\Delta_n/m^3$, where the integer $m$ is given by Lemma~\ref{lem:m}. 
Then, the integers $\lambda,\, \mu,\, P,\, Q$, and $\Delta'$ in $(\ref{eq:Delta})$ are given as follows:
\begin{align*}
\lambda &=v_2(n_2), \\
\mu & =\begin{cases}
v_3 (n_2)+3 & \text{if $3\nmid n_1$}, \\
0  & \text{if $n_1=3t\, (t\in \Z)$ and $n_2 \equiv -2t \pmod{9}$},\\
3 & \text{if $n_1=3t\, (t\in \Z),\, n_2 \not\equiv -2t \pmod{9}$ and $t\equiv n_2 \pmod{3}$},\\
2 & \text{if $n_1=3t \, (t\in \Z) $ and $t\not\equiv n_2 \pmod{3}$}, 
\end{cases}, \\
P & =\begin{cases}
\frac{e}{3} & \text{if $n_1=3t\, (t\in \Z) $ and $t\not\equiv n_2 \pmod{3}$},\\
e & \text{otherwise}, 
\end{cases} \\
Q & =\frac{ |n_2|}{2^{v_2(n_2)} 3^{v_3(n_2)}}, \\
\Delta' & =\begin{cases}
-d^2 \left(\frac{e}{3} \right)^2 & \text{if $n_1=3t\, (t\in \Z)$ and $t\not\equiv n_2 \pmod{3}$},\\
-d^2e^2 & \text{otherwise}.
\end{cases}
\end{align*}
\end{lem}
\begin{proof}
From Lemma~\ref{lem:m}, we have
\begin{align}
a & =-3\Delta_n/m^2  \label{eq:a,b} \\
& =\begin{cases}
-3de^2c & \text{if $3\nmid n_1$ or $n_1=3t \,  (t\in \Z),\, n_2 \not\equiv -2t  \!\! \pmod{9}$ and $t\equiv n_2  \!\! \pmod{3}$}, \\
-de^2c/3 & \text{otherwise},
\end{cases} \notag \\
b & = -(2n_1+3n_2)\Delta_n/m^3 \notag \\
& = \begin{cases}
-de^2 (2n_1+3n_2) & \text{if $3\nmid n_1$ or $n_1=3t \,  (t\in \Z),\, n_2 \not\equiv -2t \! \! \pmod{9}$ and $t\equiv n_2  \!\! \pmod{3}$}, \\
-de^2 (2n_1+3n_2)/3^3 & \text{otherwise},
\end{cases} \notag
\end{align}
and
\begin{align}
\Delta & =4a^3+27b^2  \label{eq:Delta=} \\
& = \begin{cases}
-3^6d^2e^4n_2^2 & \text{if $3\nmid n_1$ or $n_1=3t \, (t\in \Z),\, n_2 \not\equiv -2t  \!\! \pmod{9}$ and $t\equiv n_2 \!\! \pmod{3}$}, \\
-d^2e^4n_2^2 & \text{otherwise.}
\end{cases} \notag
\end{align}
According to Lemma~\ref{lem:m}, we have $3||e$ if $n_1=3t\, (t\in \Z)$ and $t\not\equiv n_2 \!\! \pmod{3}$, and $3\nmid e $
otherwise. Since $P$ is the largest integer divisor of $a$ such that $(P,6)=1$ and $P^2 |b$, using Lemma~\ref{lem:prime-3},
we obtain the assertion on $P$. Next, we consider the assertion on $Q$.
It follows that $n_2$ is prime to $\Delta_n =n_1^2+3n_1n_2+9n_2^2=ed^2c^3$ since $(n_1,n_2)=1$.
From $Q=\prod_{(p,6a)=1} p^{\lfloor v_p (\Delta)/2 \rfloor}$, we obtain $Q= |n_2| /(2^{v_2(n_2)} 3^{v_3(n_2)})$.
Since $\Delta'$ is the largest integer divisor of $\Delta /(PQ)^2$ satisfying $4\nmid \Delta'$ and $9\nmid \Delta'$, we
obtain the assertion on $\Delta'$. The assertions for the remaining $\lambda$ and $\mu$ are obtained from $\Delta=
2^{2\lambda} 3^{2\mu } P^2 Q^2 \Delta'$.
\end{proof}
\begin{prop}\label{prop:IB}
Let $n=n_1/n_2$ be a rational number where the integers $n_1$ and $n_2$ are coprime. Suppose that the cubic
polynomial $f_n(X)$ is irreducible over $\Q$. Let $\rho=\rho_n$ be a root of $f_n(X)$ and
$a=-3\Delta_n/m^2,\, b=-(2n_1+3n_2)\Delta_n/m^3$, where the integer $m$ is given by Lemma~\ref{lem:m}. Put
\begin{align*}
\phi & =\begin{cases}
\frac{1}{3m} (3n_2\rho -n_1-mr) & \text{if $3\nmid n_1$}, \\
& \\
\frac{1}{m} (3n_2\rho-n_1) & \text{if $3|n_1$}, 
\end{cases} \\
\psi & =\frac{1}{u} (9n_2^2 \rho^2+3n_2(mr-2n_1)\rho +m^2r^2+m^2a-mn_1r+n_1^2),
\end{align*}
\end{prop}
where $r$ and $u$ are integers given as follows. Let $\lambda,\, \mu$, and $Q$ be the integers given by Lemma~\ref{lem:lambda-mu}.
\begin{enumerate}
\item[(1)] If $3\nmid n_1$, then we have
\begin{align*}
r & =\begin{cases}
-b [(2\cdot 3^{\mu-1} \cdot a)^{\psi (Q)} 3^{\mu} +Q (2\cdot a/3 \cdot Q)^{\psi (3^{\mu-1})}] & \text{if $2\nmid n_2$}, \\
& \\
-b [(2\cdot 3^{\mu-1} \cdot a)^{\psi (Q)} 3^{\mu} +Q (2\cdot a/3\cdot Q)^{\psi (3^{\mu-1})} ] +3^{\mu-1}  aQ & \text{if $2||n_2$}, \\
& \\
-b[ (2^{\lambda+1} \cdot 3^{\mu-1}\cdot a)^{\psi (Q)} 3^{\mu} +Q(2^{\lambda+1} \cdot a/3 \cdot Q)^{\psi (3^{\mu -1})}] 2^{\lambda} \\
\quad +(2^{\lambda-1}-3\cdot b/2)(3^{\mu-1} aQ)^{\psi (2^{\lambda})} Q\cdot 3^{\mu -1}  & \text{ if $4|n_2$}, 
\end{cases} \\
u & =9ec^2n_2.
\end{align*}
\item[(2)] If $3|n_1$, then we have
\begin{align*}
r & = \begin{cases}
-b[3(2a)^{\psi (Q)}+Q^2]  & \text{if $2\nmid n_2$},\\
& \\
-b [3(2a)^{\psi (Q)}+Q^2]+3Q^2a  & \text{if $2||n_2$},\\
& \\
-3 \cdot 2^{\lambda}\cdot b (2^{\lambda+1} a)^{\psi (Q)} +3Q (2^{\lambda-1} -3\cdot b/2 )(3aQ)^{\psi (2^{\lambda})}-2^{2\lambda}bQ^2 & \text{if $4|n_2$},
\end{cases}\\
u & = \begin{cases}
9ec^2n_2 & \text{if $n_1=3t\, (t\in \Z)$ and $n_2 \equiv -2t \pmod{9}$},\\
& \\
3ec^2n_2& \text{otherwise}.
\end{cases}
\end{align*}
Then, $\{1,\phi, \psi \}$ is an integral basis of $L_n$.
\end{enumerate}
\begin{proof}
(1)\ Assume $3\nmid n_1$. According to Lemmas \ref{lem:equiv-0,1}, \ref{lem:m} (1), and \ref{lem:lambda-mu}, and
(\ref{eq:a,b}), we have
\begin{align*}
a & =-3de^2c \equiv 6 \pmod{9}, \\
b & =-de^2 (2n_1+3n_2) \not\equiv 0 \pmod{3},\\
\mu &>2.
\end{align*}
Therefore, from the results of Albert, an integral basis of $L_n$ is given by (\ref{eq:Albert-IB1}).
If $2\nmid n_2$, then since $2\nmid \Delta_n$ from Lemma~\ref{lem:equiv-0,1}, we have $b\not\equiv 0 \pmod{2}$. Therefore,
the integers $\varepsilon_2$ and $r$ in (\ref{eq:Albert-IB1}) are given by (\ref{eq:Albert-1-4}). 
Next, assume $2||n_2$. Since $2\nmid c $ from Lemma~\ref{lem:equiv-0,1}, and hence $c^2 \equiv 1 \pmod{4}$,
we have
$a=-3de^2c\equiv de^2c^3 \equiv \Delta_n \equiv 1+3n_1n_2 \not\equiv 1 \pmod{4}$ and $b\equiv 0\pmod{4}$.
Therefore, the integers $\varepsilon_2$ and $r$ in (\ref{eq:Albert-IB1}) are given by (\ref{eq:Albert-1-3}).
Finally, assume $4|n_2$. Since $2\nmid c$ from Lemma~\ref{lem:equiv-0,1}, and hence $c^2\equiv 1\pmod{4}$,
we have $a=-3de^2c\equiv de^2c^3 \equiv \Delta_n \equiv n_1^2 \equiv 1 \pmod{4}$ and $b=-de^2 (2n_1+3n_2)
\equiv 2 \pmod{4}$. Furthermore, according to Lemma~\ref{lem:lambda-mu}, we have
$\Delta' =-de^2 \equiv 3\pmod{4}$.
Therefore, the integers $\varepsilon_2$ and $r$ in (\ref{eq:Albert-IB1}) are given by (\ref{eq:Albert-1-1}).
From Lemma~\ref{lem:lambda-mu} and (\ref{eq:theta}), we get the assertion in the case of $3\nmid n_1$ (we obtain an integral basis
$\{1,\phi,-\psi \}$ in the case of  $n_2 <0$ from the results of Albert; however $\{1, \phi, \psi \}$ is also an integral basis).

(2)\ Assume $3|n_1$ and put $n_1=3t\, (t\in \Z)$. We show that $a\equiv 6 \pmod{9},\, b\not\equiv 0 \pmod{3}$ and 
$\mu >2$ are not simultaneously satisfied in this case. If $n_2 \not\equiv -2t \pmod{9}$ and $t\equiv n_2 \pmod{3}$, then
we have $a=-3de^2c\equiv 0 \pmod{9}$ from Lemma~\ref{lem:m} (3) (i) and (\ref{eq:a,b}).
If $t\not\equiv n_2 \pmod{3}$, then we have $b=-de^2 (2n_1+3n_2)/3^3 =-d(e/3)^2 (2t+n_2) \not\equiv 0\pmod{3}$ from Lemmas~\ref{lem:equiv-0,1},
\ref{lem:m} (3) (ii), and (\ref{eq:a,b}). If $n_2 \equiv -2t \pmod{9}$, then we have $\mu=0$ from Lemma~\ref{lem:lambda-mu}.
Therefore, from the results of Albert, an integral basis of $L_n$ is given by (\ref{eq:Albert-IB2}). Next, we determine $\varepsilon_1$ in 
(\ref{eq:Albert-2-1}).
If $n_2 \not\equiv -2t \pmod{9}$ and $t\equiv n_2 \pmod{3}$, then we have $a=-3de^2c\equiv 0 \pmod{3}$ and $b=
-de^2(2n_1+3n_2) =-3de^2(2t+n_2)\equiv 0 \pmod{9}$ from (\ref{eq:a,b}).
Therefore, we have
$\varepsilon_1=1$. If $t\not\equiv n_2 \pmod{3}$, then we have $a=-de^2c/3=-3d(e/3)^2c\equiv 6 \pmod{9}$ and
$b=-de^2 (2n_1+3n_2)/3^3=-d(e/3)^2 (2t+n_2)\not\equiv 0\pmod{9}$ from Lemmas~\ref{lem:equiv-0,1}, \ref{lem:m} (3) (ii) and
(\ref{eq:a,b}). We shall show that $b^2+a-1 \not\equiv 0\pmod{9}$.
We have $b^2+a-1 \equiv b^2-4 \equiv (b+2)(b-2) \pmod{9}$.
Since $3\nmid c$, we have $c^3\equiv \pm 1 \pmod{9}$ and hence
\begin{align*}
b\pm 2 & \equiv \pm d \left(\frac{e}{3} \right)^2 c^3 (2t+n_2)\pm 2 \pmod{9} \\
& \equiv \pm \frac{1}{3^2} (n_1^2+3n_1n_2+9n_2^2)(2t+n_2)\pm 2\\
& \equiv \pm (t^2+tn_2+n_2^2)(2t+n_2)\pm 2 \\
& \equiv \pm [2t^3+3tn_2(t+n_2)+n_2^3)] \pm 2.
\end{align*}
If $3\nmid t$, then we have $t+n_2 \equiv 0\pmod{3}$ from $t\not\equiv n_2 \pmod{3}$ and
$3\nmid n_2$.
We obtain $b\pm2 \not\equiv 0 \pmod{9}$ since
\[
b\pm 2 \equiv \begin{cases}
\pm n_2^3 \pm 2 \pmod{9} & \text{if $3|t$}, \\
\pm t^3\pm 2 \pmod{9 } & \text{if $3\nmid t$}.
\end{cases}
\]
Since $b+2 \equiv b-2 \equiv 0\pmod{3}$ does not hold, we have $b^2+a-1 \equiv (b+2)(b-2) \not\equiv 0
\pmod{9}$. From (\ref{eq:Albert-2-1}), we have $\varepsilon_1=0$. If $n_2\equiv -2t \pmod{9}$, then we have
$a=-de^2(c/3) \equiv -1 \pmod{3}$ from Lemmas~\ref{lem:equiv-0,1}, \ref{lem:m} (2) and (\ref{eq:a,b}).
If $b\not\equiv 0\pmod{3}$,
then we have $b^2+a-1 \not\equiv 0\pmod{3}$. Therefore, we have $\varepsilon_1=0$. Combining these,
we obtain the following.
\[
\varepsilon_1  =\begin{cases}
1 & \text{if $n_2 \not\equiv -2t \pmod{9}$ and $t\equiv n_2 \pmod{3}$}, \\
0 & \text{otherwise}.
\end{cases}
\]
Next, we determine $\varepsilon_2$ and $r$. If $2\nmid n_2$, then we have $b=-(2n_1+3n_2)\Delta_n/m^3 \not\equiv 0\pmod{2}$
from Lemma~\ref{lem:equiv-0,1}, and hence $\varepsilon_2$ and $r$ are given by (\ref{eq:Albert-2-5}).
If $2||n_2$, then since $2n_1+3n_2 \equiv 0 \pmod{4}$ and $2\nmid m$, we have $b=-(2n_1+3n_2)\Delta_n/m^3 \equiv 0 \pmod{4}$.
Furthermore, we have
$a\equiv \Delta_n \equiv de^2c^3 \equiv n_1^2+3n_1n_2+9n_2^2\equiv 1+3n_1n_2\not\equiv 1\pmod{4}$.
The integers $\varepsilon_2$ and $r$ are given by (\ref{eq:Albert-2-4}) in this case. Finally, we consider the case $4|n_2$.
In this case, we have $2||(2n_1+3n_2)$, and since $2\nmid m$, we  have $2|| b =-(2n_1+3n_2)\Delta_n/m^3$.
Furthermore, we have $a\equiv de^2c^3\equiv n_1^2+3n_1n_2+9n_2^2 \equiv n_1^2\equiv 1 \pmod{4}$,
and $\Delta' \equiv \Delta_n \equiv 3 \pmod{4}$ from Lemma~\ref{lem:lambda-mu}. The integers $\varepsilon_2$ and $r$ are given by
(\ref{eq:Albert-2-2}) in this case. Combining all these and Lemmas~\ref{lem:m}, \ref{lem:lambda-mu}, (\ref{eq:theta}), we get the assertions (we obtain an integral basis
$\{ 1,\phi, -\psi \}$ in the case of $n_2<0$ from the results of Albert; however $\{1,\phi,\psi \}$ is also an integral basis).
\end{proof}
\begin{rem}\label{rem:Komatsu}
Komatsu \cite[Corollary~2.8 and Lemma~4.1]{K} gave an integral basis of a cyclic cubic field $L$ using the rational number with
the smallest weil height among the rational numbers $s$ satisfying $L=L_{3s}$.
\end{rem}

For $x_1, x_2, x_3 \in L_n$, we define
the discriminant:
\[
d(x_1,x_2,x_3):= \begin{vmatrix}
x_1 & x_2 & x_3 \\
x_1' & x_2' & x_3'  \\
x_1'' & x_2'' & x_3'' \\
\end{vmatrix}^2.
\]
The discriminant $d(f_n)$ of $f_n(X)$ coincides with $d(1,\rho_n, \rho_n^2)$.
For an integral basis $\{1, \phi, \psi \}$ of $L_n$, let
\[ 
D_{L_n} :=d (1,\phi,\psi)
\]
be the discriminant of $L_n$. It is known from the conductor-discriminant formula that $D_{L_n}=\mathfrak f_{L_n}^2$,
where $\mathfrak f_{L_n}$ is the conductor of $L_n$ (\cite[Theorem~3.11]{W1}). From Proposition~\ref{prop:IB} and 
$d(1,\rho_n,\rho_n^2)=\Delta_n^2/n_2^4$, we obtain the discriminant of $L_n$. If $n$ is an integer, then we have $n_2= \pm 1$, and
the conductor  is known (see Cusick \cite[Lemma~1]{C} when $\Delta_n$ is square-free, Washington \cite[Proposition~1]{W2}
when $n\not\equiv 3 \pmod{9}$, and Kashio-Sekigawa \cite{KS} in general). It was also obtained for a rational number $n$ by Komatsu \cite[Lemma~1.4]{K}
although the formulation is slightly different.
\begin{cor}\label{cor:D}
We have the following:
\begin{enumerate}
\item[(1)] If $3\nmid n_1$, then we have $D_{L_n} =d^2e^2$.
\item[(2)] If $n_1=3t\, (t\in \Z)$ and $n_2 \equiv -2t \pmod{9}$, then we have $D_{L_n} =d^2e^2$.
\item[(3)] If $n_1=3t\, (t\in \Z)$ and $n_2 \not\equiv -2t \pmod{9}$, then we have the following.
\begin{enumerate}
\item[(i)] If $t\equiv n_2 \pmod{3}$, then $D_{L_n}=3^4d^2e^2$.
\item[(ii)] If $t\not\equiv n_2 \pmod{3}$, then $D_{L_n}=3^2 d^2e^2$.
\end{enumerate}
\end{enumerate}
\end{cor}

From the corollary and Lemma~\ref{lem:m}, it follows that $L_n/\Q$ is tamely ramified in the cases of (1) and (2) of Corollary~\ref{cor:D},
and wildely ramified in the case of (3).
\section{Branch classes and associated orders}\label{sec:BA}

In this  section, we consider the Galois module structure of the ring of integers of an abelian number field
that was obtained by Leopoldt \cite{Leo}.
We refer to  a well-written article \cite{Le} by Lettl that simplifies Leopoldt's work.
Let $L$ be an abelian number field with  Galois group $G:=\Gal(L/\Q)$ 
and conductor $\mathfrak f$. Let $\frakX$  be the group of Dirichlet characters  associated to $L$.
Let $v_p(x)$  denote the normalized $p$-adic valuation of $x \in \Q$ for a prime number $p$.
For any $m\in \mathbb Z_{>0}$, we put
\[
p (m) =\prod_{\substack{ 
p\mid m
\\
p\neq 2
}} p,\qquad 
q(m) = \prod_{\substack{
p \\
v_p(m) \geq 2}} p^{v_p(m)},
\]
where the first product runs over all odd prime numbers $p$ dividing $m$, and the second product runs over all
prime numbers $p$ that satisfy $v_p (m)\geq 2$.
Put
\[
\calD (\mathfrak f)= \{ m\in \Z_{>0} \, |\, p(\mathfrak f) | m,\ m| \mathfrak f,\ m\not\equiv 2 \pmod{4}\}.
\]
 We define a {\it branch class} of $\frakX$ for any $m\in \calD (\mathfrak f)$ as follows:
\[
\Phi_m := \{ \chi \in \frakX \, |\, q(\mathfrak f_{\chi})=q(m) \},
\]
where $\mathfrak f_{\chi}$ is the conductor of $\chi$.
We have $\frakX=\coprod_{m\in \calD (\mathfrak f)} \Phi_m$ (disjoint union). For any $\chi \in \frakX$,
let
\[
{\bf e}_{\chi} = \frac{1}{[L:\Q]} \sum_{\sigma \in G} \chi^{-1} (\sigma ) \sigma 
\]
be the idempotent. Furthermore, for any $m \in \calD (\mathfrak f)$, let
\[
{\bf e}_m = \sum_{\chi \in \Phi_m} {\bf e}_{\chi}.
\]
Since the branch class $\Phi_m$ is closed under conjugation,
we obtain ${\bf e}_m \in \Q [G]$.
Let $\calO_L$ be the ring of integers of $L$.
The action of the group ring $\Q[G]$ on $L$ is given as follows:
\[
x\, a :=\sum_{a\in G} n_{\sigma} \sigma (a) \ \left( x= \sum_{\sigma \in G} n_{\sigma} \sigma \in \Q [G],\, a\in L \right).
\]
We define  {\it the associated order} for $L/\Q$ as follows:
\[
\calA_{L/\Q}:= \{ x\in \Q [G] \, |\, x \, \calO_{L} \subset \calO_{L} \}.
\]
Leopoldt showed that $\calO_{L}$ is a free module of rank $1$ over $\calA_{L/\Q}$.
The generator $T \, (\in \calO_L)$ is given by
\[
T=\sum_{m \in \calD (\mathfrak f)} \eta_m,\quad \eta_m:= \Tr_{\Q (\zeta_m)/(L\cap \Q(\zeta_m))} (\zeta_m) \, \text{ (Gaussian period)}.
\]
\begin{theo}[Leopoldt \cite{Leo}]\label{theo:Leopoldt}
Let $L$ be an abelien number field with Galois group $G$. The ring of integers $\calO_L$ is a free module of rank $1$ over $\calA_{L/\Q}$, and we have
\[
\calO_L=\bigoplus_{m \in \calD (\mathfrak f)} \Z [G]\, \eta_m =\calA_{L/\Q}\,  T
\]
and
\[
\calA_{L/\Q}=\Z [G][\{{\bf e}_m\, |\, m\in \calD (\mathfrak f) \}].
\]
\end{theo}

Let $n=n_1/n_2$ be a rational number where the integers $n_1$ and $n_2$ are coprime.
Suppose that the cubic polynomial $f_n(X)$ is irreducible over $\Q$.
Let $L_n$ be the cubic field defined in \S \ref{sec:intro}. We express as in \S \ref{sec:intro} that 
$\Delta_n=n_1^2+3n_1n_2+9n_2^2=de^2c^3$ using $d,e,c \in \Z_{>0}$, where $d$ and $e$ are square-free
and $(d,e)=1$. According to Corollary~\ref{cor:D} and $D_{L_n}=\mathfrak f_{L_n}^2$, the conductor $\mathfrak f_{L_n}$
is given as follows.

\begin{align}
& \text{(1)\ If $3\nmid n_1$, then we have $\mathfrak f_{L_n}=de$.} \label{eq:f} \\
& \text{(2)\ If $n_1=3t\, (t\in \Z)$ and $n_2\equiv -2t \pmod{9}$, then we have $\mathfrak f_{L_n}=de$.} \notag \\
& \text{(3)\ If $n_1=3t\, (t\in \Z)$ and $n_2 \not\equiv -2t \pmod{9}$, then we have the following. } \notag \\
& \qquad \text{(i)\ If $t\equiv n_2 \pmod{3}$, then $\mathfrak f_{L_n}=3^2de$.} \notag \\
& \qquad  \text{(ii)\ If $t\not\equiv n_2 \pmod{3}$, then $\mathfrak f_{L_n}=3de$.}\notag
\end{align}

From Lemma~\ref{lem:m}, the set $\calD (\mathfrak f_{L_n})$ is given by
\begin{equation}\label{eq:D(f)}
\calD (\mathfrak f_{L_n }) =\begin{cases}
\{ 9de, 3de \} & \text{if $n_1=3t\, (t\in \Z),  n_2 \not\equiv -2t \!\! \pmod{9}$ and $t\equiv n_2 \!\! \pmod{3}$}, \\
\{ 3de, de \} &  \text{if $n_1=3t\, (t\in \Z)$ and $t\not\equiv n_2 \!\! \pmod{3}$}, \\
\{ de \} & \text{otherwise},
\end{cases}
\end{equation}
and hence the associated order $\calA_{L_n/\Q}$ and the generator $T$ of $\calO_{L_n}$ are given by
\begin{equation}\label{eq:A}
\calA_{L_n/\Q}=\begin{cases}
\Z [G] [{\bf e}_{9de}, {\bf e}_{3de}] & \text{if $n_1=3t\, (t\in \Z),  n_2 \not\equiv -2t \!\! \pmod{9}$ and $t\equiv n_2 \!\! \pmod{3}$}, \\
\Z [G] [{\bf e}_{3de}, {\bf e}_{de}] &  \text{if $n_1=3t\, (t\in \Z)$ and $t\not\equiv n_2 \!\! \pmod{3}$}, \\
\Z [G][{\bf e}_{de} ] &  \text{otherwise},
\end{cases}
\end{equation}
\begin{equation}\label{eq:T}
T= \sum_{m \in \calD (\mathfrak f_{L_n} )} \eta_m =\begin{cases}
\eta_{9de}+\eta_{3de}& \text{if $n_1=3t\, (t\in \Z),  n_2 \not\equiv -2t \!\! \pmod{9}$ and $t\equiv n_2 \!\! \pmod{3}$}, \\
\eta_{3de}+\eta_{de} &   \text{if $n_1=3t\, (t\in \Z)$ and $t\not\equiv n_2 \!\! \pmod{3}$}, \\
\eta_{de}&  \text{otherwise}.
\end{cases}
\end{equation}

\section{Tamely ramified cases}\label{sec:tame}
Let $n_1/n_2$ be a rational number where the integers $n_1$ and $n_2$ are coprime.
Suppose that the cubic polynomial $f_n(X)$ is irreducible over $\Q$ and consider the cases (1) and (2)
of (\ref{eq:f}). In these cases, $L_n/\Q$ is tamely ramified since the conductor of $L_n$ is
$\mathfrak f_{L_n}=de$  is square-free.
The corresponding group of Dirichlet characters is $\frakX=\Phi_{de} =
\{ {\bf 1}, \chi, \chi^2 \}$, where $\bf 1$ is the trivial character and $\chi$ is a character given by $\chi : G \to \C^{\times},\
\sigma \mapsto \zeta_3$. Since ${\bf e}_{de}=1$, we have $\calA_{L_n/\Q}=\Z [G][{\bf e}_{de}]=\Z [G]$ and
$\calO_{L_n}=\Z [G]  \eta_{de}$.
The ring of integers $\calO_{L_n}$ is isomorphic to $\Z [G]$ as $\Z [G]$-module, and the conjugates $\{
\eta_{de}, \eta_{de}', \eta_{de}'' \}$ of $\eta_{de}$ is a normal integral basis of $L_n$.
 Put $\rho := \rho_n$, $\zeta := \zeta_3$ and  $\Delta_n =n_1^2+3n_1n_2+9n_2^2=A_n A_n'$ with
$A_n:=n_1+3n_2(1+\zeta),\, A_n':=n_1+3n_2(1+\zeta^2)$.
In these cases, the generator of the $\Z [G]$-module $\calO_{L_n}$ is  obtained using roots of $f_n(X)$ as follows.
\begin{theo}\label{theo:tame}
Let $n=n_1/n_2$ be a rational number where the integers $n_1$ and $n_2$ are coprime.
Suppose that the cubic polynomial $f_n(X)$ is irreducible over $\Q$, and $3\nmid n_1$ or
$n_1=3t\, (t\in \Z),\, n_2 \equiv -2t \pmod{9}$. 
There exist integers $a_0$ and $a_1$  that satisfy $ec=a_0^2-a_0a_1+a_1^2$, and
$a_0+a_1\zeta $ divides $ A_n$ in $\Z [\zeta]$.
Furthermore, for any such $a_0$ and $a_1$, we  have $(\varepsilon ec^2-n_1(a_0+a_1))/3\in \Z$ for $\varepsilon \in \{ \pm 1\} $ given by
\[
\varepsilon: =
\begin{cases}
\left( \frac{n_1 (a_0+a_1)}{3} \right) & \text{if $3\nmid n_1$}, \\
\left( \frac{n_2a_0}{3} \right) & \text{if $n_1=3t\, (t\in \Z),\, n_2\equiv -2t \!\! \pmod{9}$}, 
\end{cases}
\]
where $\left( \frac{\cdot}{3} \right)$ is the Legendre symbol, and 
\begin{equation*}
\alpha :=\frac{1}{ec^2}\left( n_2(a_0\rho_n+a_1\rho_n')+\frac{1}{3} (\varepsilon ec^2-n_1(a_0+a_1) ) \right)
\end{equation*}
is a generator of a normal integral basis of the cyclic cubic field $L_n$, namely, we have
$\calO_{L_n}=\Z [G]  \, \alpha =\calA_{L_n/\Q}\, \alpha$.
\end{theo}

\begin{rem}\label{rem:tame}
\begin{enumerate}
\item[(1)] Since $\Z [G]^{\times} =\{ \pm 1_G, \pm \sigma, \pm \sigma^2\}$,
a normal integral basis 
$ \{ \eta_{de}, \eta_{de}' ,\eta_{de}'' \}$ 
coincides with $ \{ \alpha, \alpha', \alpha'' \}$ or $\{ -\alpha,
-\alpha', -\alpha'' \}$, where $\alpha$ is the element defined in the theorem above.
\item[(2)] The  result in the cases of the simplest cubic fields (namely, when $n_2=\pm 1$) has been obtained
in \cite{HA1}. Furthermore, the result for Lehmer's cyclic quintic fields has been obtained in \cite{HA2}.
\end{enumerate}
\end{rem}
\begin{proof}
Let $\{1,\phi, \psi \}$ be the integral basis of $L_n$ given by Proposition~\ref{prop:IB}. We find $\alpha$ satisfying 
$\calO_{L_n}=\Z [G]  \alpha$ using the method proposed by Acciaro and Fieker \cite{AF}. From $\rho+\rho'+\rho'' =n$
and $\rho^2=\rho'+(n+1)\rho +2$, we obtain the following:
\begin{align}
1 & =\frac{n_2}{n_1} (\rho+\rho'+\rho'' ), \notag \\
\phi & =\begin{cases}
\dfrac{n_2}{3cn_1} (3n_1 \rho-(n_1+cr)(\rho+\rho'+\rho'')) & \text{ in the case of (\ref{eq:f}),(1)}, \\
& \\
\dfrac{n_2}{3c} (3\rho-(\rho+\rho'+\rho'')) & \text{ in the case of  (\ref{eq:f}),(2)}, 
\end{cases} \label{eq:IB-tame} \\
\psi & =\frac{1}{9ec^2n_1} (9 (n_1n_2 \rho' +n_1 (n_1+n_2)\rho +2n_2^2 (\rho+\rho'+\rho'')) \notag \\
 & \quad +3n_1 (mr-2n_1)\rho +(m^2r^2+m^2a-mn_1r+n_1^2)(\rho+\rho'+\rho'')), \notag
\end{align}
where
\[
m=\begin{cases}
c & \text{ in the case of (\ref{eq:f}),(1)}, \\
3c & \text{ in the case of (\ref{eq:f}),(2)}.
\end{cases}
\]
Put $\ell =9ec^2n_1$. We have
\[
1=g_1\, \frac{\rho}{\ell},\quad \phi =g_2\, \frac{\rho}{\ell}, \quad \psi=g_3\, \frac{\rho}{\ell},
\]
where $g_1,g_2,g_3 \in \Z [G]$ are given by
\begin{align}
g_1 & =\frac{\ell n_2}{n_1} (1+\sigma+\sigma^2), \notag \\
g_2 & =\begin{cases}
\dfrac{\ell n_2}{3cn_1} (3n_1-(n_1+cr)(1+\sigma+\sigma^2)) & \text{  in the case of (\ref{eq:f}),(1)}, \\
 & \\
\dfrac{\ell n_2}{3c} (3-(1+\sigma+\sigma^2)) & \text{  in the case of (\ref{eq:f}),(2)}, 
\end{cases} \label{eq:g-tame} \\
g_3 & =\frac{\ell}{9ec^2n_1} (9(n_1n_2 \sigma+n_1(n_1+n_2)+2n_2^2(1+\sigma+\sigma^2)) \notag\\
& \quad +3n_1 (mr-2n_1)+(m^2r^2 +m^2a-mn_1r+n_1^2)(1+\sigma+\sigma^2)). \notag
\end{align}
Since $\{ 1,\phi, \psi \}$ is an integral basis, we obtain from (\ref{eq:g-tame}) that
\begin{equation}\label{eq:O-tame}
\calO_{L_n}=\Z+\phi \Z+\psi \Z =(g_1\Z[G]+g_2 \Z[G]+g_3 \Z [G])\frac{\rho}{\ell}.
\end{equation}
Since $\calO_{L_n}=\Z [G] \eta_{de}$, there exists $g\in \Z [G]$ satisfying
\begin{equation}\label{eq:eta-tame}
g \, \frac{\rho}{\ell}=\eta_{de}.
\end{equation}
Therefore, we have 
\begin{equation}\label{eq:eta=g-tame}
\calO_{L_n}=g\Z[G] \, \frac{\rho}{\ell} =\Z [G] \,  \eta_{de}.
\end{equation}
From (\ref{eq:O-tame}) and (\ref{eq:eta=g-tame}), we obtain the equality of ideals of $\Z [G]$:
\[
(g)=\mathrm{Ann}_{\Z [G]} \left( \frac{\rho}{\ell} \right) =(g_1,g_2,g_3 )+\mathrm{Ann}_{\Z [G]} \left( \frac{\rho}{\ell} \right).
\]
It follows from $n\ne 0$ that $\{ \rho/\ell, \rho'/\ell, \rho''/\ell \}$ is a normal basis (\cite[Proof of Theorem~4.4]{HA1}).
Therefore, we have $\mathrm{Ann}_{\Z [G]}(\rho/\ell)=0$, and we obtain the equality of ideals of $\Z [G]$:
\begin{equation}\label{eq:g123-tame}
(g)=(g_1,g_2,g_3).
\end{equation}
Consider the surjective ring homomorphism 
\[
\nu: \Z [G] \longrightarrow \Z [\zeta]
\]
defined by $\nu (\sigma )=\zeta$. We calculate the image of the ideal $I:=(g)=(g_1,g_2,g_3)$ by $\nu$.
Since $\nu$ is surjective, we obtain the ideal of $\Z [\zeta]$:
\begin{equation}\label{eq:nu123-tame}
\nu (I)=(\nu (g))=(\nu (g_1), \nu (g_2), \nu (g_3)).
\end{equation}
From (\ref{eq:g-tame}), the elements $\nu (g_1), \nu (g_2), \nu (g_3)$ are given by
\begin{align*}
\nu (g_1) & =\dfrac{\ell n_2}{n_1} (1+\zeta+\zeta^2)=0, \\
\nu (g_2) & =\begin{cases}
\dfrac{\ell n_2}{3cn_1} (3n_1-(n_1+cr)(1+\zeta+\zeta^2)) & \text{ in the case of (\ref{eq:f}),(1).} \\
& \\
\dfrac{\ell n_2}{3c} (3- (1+\zeta+\zeta^2)) & \text{ in the case of (\ref{eq:f}),(2)} 
\end{cases} \\
& =9ecn_1n_2, \\
\nu (g_3) & =\dfrac{\ell}{9ec^2n_1} (9(n_1n_2 \zeta +n_1(n_1+n_2)+2n_2^2 (1+\zeta+\zeta^2)) \\
& \quad +3n_1 (mr -2n_1)+(m^2 r^2 +m^2 a-mn_1 r+n_1^2)(1+\zeta+\zeta^2)) \\
& =3n_1 (A_n +mr),
\end{align*}
where $A_n=n_1+3n_2 (1+\zeta)$. Therefore, we obtain
\begin{equation}\label{eq:nuI-tame}
\nu (I)=( \nu (g))=3n_1 (3ecn_2 , A_n+mr).
\end{equation}
Let $p_1,\ldots, p_k$ be different prime numbers satisfying
\begin{equation}\label{eq:p1k-tame}
p_1^{\iota_1} \cdots p_k^{\iota_k} =\begin{cases}
ec & \text{ in the case of (\ref{eq:f}),(1)}, \\
e\cdot\frac{c}{3} & \text{ in the case of (\ref{eq:f}),(2)},
\end{cases}
\end{equation}
with $\iota_1,\ldots,\iota_k \in \Z_{>0}$. For
$i\in \{1,\ldots, k \}$, we put $p_i =\pi_i \pi_i'$, where $\pi_i$ and $\pi_i'$ are prime elements
of $\Z [\zeta]$ satisfying $ \pi_i |A_n$ and $\pi_i' |A_n'$. 
Let $a_0 $ and $a_1$ be integers satisfying 
\begin{equation}\label{eq:a0a1-tame}
a_0 +a_1 \zeta =\begin{cases}
\pi_1^{\iota_1} \cdots \pi_k^{\iota_k} & \text{ in the case of (\ref{eq:f}),(1)}, \\
(1-\zeta) \pi_1^{\iota_1} \cdots \pi_k^{\iota_k}  & \text{ in the case of (\ref{eq:f}),(2)}. 
\end{cases}
\end{equation}
Then $a_0 $ and $a_1$ satisfy
\begin{equation}\label{eq:df-a0a1-tame}
ec=a_0^2-a_0a_1+a_1^2=(a_0+a_1\zeta)(a_0+a_1\zeta^2),\quad a_0+a_1\zeta \, | \, A_n.
\end{equation}
Conversely, any integers $a_0$ and $a_1$ satisfying (\ref{eq:df-a0a1-tame}), $a_0+a_1\zeta$ is given by the right-hand side of (\ref{eq:a0a1-tame}) up to a unit of
$\Z [\zeta]$.

Next, we shall prove 
\begin{equation}\label{eq:nuI-gen-tame}
\nu (I)=9n_1n_2 (a_0 +a_1 \zeta).
\end{equation}
%where $a_0$ and $a_1$ are integers that satisfy $ec=a_0^2-a_0a_1+a_1^2 $, and $a_0 +a_1 \zeta $ divides $A_n$ in $\Z [\zeta]$.
Put
\begin{equation}\label{eq:delta-tame}
(3ecn_2, A_n+mr)=(\delta ), \quad \delta \in \Z [\zeta]
\end{equation}
and consider the $\mathfrak p$-adic valuations of $\delta$ for prime ideals $\mathfrak p$ of $\Z [\zeta]$.
Let $v_p (x)$ and $v_{\mathfrak p}(x) $ be the normalized $p$-adic valuation of $x\in \Q$ for a prime number $p$ and the normalized 
$\mathfrak p$-adic valuation of $x\in \Q (\zeta)$ for a prime ideal $\mathfrak p$ of $\Z [\zeta]$, respectively.
Obviously, we have $v_{\mathfrak p}(\delta )=0$ if $v_{\mathfrak p}(3ecn_2)=0$.

First, we consider the $2$-adic valuation of $\delta$.
Let $\mathfrak p =(2)$ be the prime ideal of $\Z [\zeta]$ above $2$. We shall show
\begin{equation}\label{eq:delta2-tame}
v_{\mathfrak p} (\delta )=v_2(n_2).
\end{equation}
We have $2\nmid ec$ from Lemma~\ref{lem:equiv-0,1}, and hence $v_2(3ecn_2)=v_2(n_2)$.
Therefore, (\ref{eq:delta2-tame}) follows from
\begin{equation}\label{eq:A2-tame}
A_n +mr =n_1+3n_2(1+\zeta)+mr \equiv 0 \pmod{ (2^{v_2(n_2)})}.
\end{equation}
We shall show this congruence.
If $2\nmid n_2$, then (\ref{eq:A2-tame}) holds.
We consider the case $2||n_2$.
From Lemmas~\ref{lem:equiv-0,1}, \ref{lem:lambda-mu} and (\ref{eq:a,b}), we have $2\nmid a,\, 2|b$ and $2\nmid Q$.
Furthermore, we have $2\nmid m$ from Lemma~\ref{lem:m}.
According to Proposition~\ref{prop:IB}, we have
\begin{align*}
r & \equiv \begin{cases}
3^{\mu -1}aQ & \text{ in the case of (\ref{eq:f}), (1)}, \\
3Q^2a & \text{ in the case of (\ref{eq:f}), (2) },
\end{cases} \\
& \equiv 1 \pmod{2}.
\end{align*}
Combining all these, we have $A_n +mr \equiv 0 \pmod{ (2)}$,  and hence (\ref{eq:A2-tame}) holds in  the case of
$2 ||n_2$. Next, we consider the case of $4|n_2$. In both cases (1) and (2) of (\ref{eq:f}),
according to Proposition~\ref{prop:IB} and Lemma~\ref{lem:lambda-mu}, we have $r\equiv (2^{\lambda-1} -3\cdot b/2)a^{\psi (2^{\lambda})}
\pmod{2^{\lambda}}$ and $\lambda =v_2(n_2)$.
In the case of (\ref{eq:f}), (1), we have $m=c$, and from (\ref{eq:a,b}), we have
\begin{align*}
ac^2 (A_n+mr) & \equiv ac^2 (n_1+3n_2(1+\zeta)) +c^3 (2^{\lambda-1} -3\cdot b/2) \\
& \equiv -3n_1 \Delta_n +2^{\lambda-1} c^3 +\dfrac{3}{2} (2n_1+3n_2)\Delta_n \\
& \equiv -3n_1^2+2^{\lambda-1} c^3 +\dfrac{3}{2} (2n_1+3n_2)n_1^2 \\
& \equiv 2^{\lambda -1} \left( c^3 +9n_1^2 \cdot \dfrac{n_2}{2^{\lambda}} \right) \\
& \equiv 0 \pmod{ (2^{\lambda}) }.
\end{align*}
Since $2\nmid ac^2$, (\ref{eq:A2-tame}) holds in the case of $4|n_2$ and (\ref{eq:f}), (1).
In the case of $4|n_2$ and (\ref{eq:f}),(2), we have
$m=3c$, and from (\ref{eq:a,b}), we have
\begin{align*}
a(A_n+mr) & \equiv a(n_1+3n_2 (1+\zeta))+3c (2^{\lambda-1} -3\cdot b/2) \\
& \equiv -de^2 \dfrac{c}{3}n_1+3\cdot 2^{\lambda-1}c +\dfrac{1}{2\cdot 3} (2n_1+3n_2) de^2c \\
& \equiv 2^{\lambda -1} c \left( 3+\dfrac{n_2}{2^{\lambda}} de^2 \right) \\
& \equiv 0 \pmod{ (2^{\lambda})}.
\end{align*}
Since $2\nmid a$, (\ref{eq:A2-tame}) holds in the case of $4|n_2$ and (\ref{eq:f}),(2).
Hence, we obtain (\ref{eq:delta2-tame}) on the $2$-adic valuation of $\delta$.

Next,we consider the $3$-adic valuation of $\delta$. Let $\mathfrak p =(1-\zeta)$ be the prime ideal of $\Z [\zeta]$
above $3$. We shall show 
\begin{equation}\label{eq:delta3-tame}
v_{\mathfrak p} (\delta )=
\begin{cases}
2(1+v_3(n_2)) & \text{ in the case of (\ref{eq:f}), (1)},\\
3 & \text{ in the case of (\ref{eq:f}), (2) }.
\end{cases}
\end{equation}
Note that $3\nmid n_2$ in the case of (\ref{eq:f}),(2). From Lemma~\ref{lem:m}, we have
\[
v_{\mathfrak p} (3ecn_2)=2v_3(3ecn_2) =\begin{cases}
2(1+v_3(n_2)) & \text{ in the case of (\ref{eq:f}), (1)},\\
4 & \text{ in the case of (\ref{eq:f}), (2) }.
\end{cases}
\]
Therefore, (\ref{eq:delta3-tame}) follows from
\begin{equation}\label{eq:A3-tame}
v_{\mathfrak p}(A_n+mr)=2v_3(n_2)+3.
\end{equation}
We shall show the equality. In the case of (\ref{eq:f}),(1), according to Proposition~\ref{prop:IB}, Lemma~\ref{lem:lambda-mu}, and
(\ref{eq:a,b}), we have $r\equiv -b (2\cdot a/3)^{\psi (3^{\mu -1})} \equiv -(2n_1+3n_2)(2c)^{\psi (3^{\mu -1})} \pmod{ 3^{\mu -1} }$
and $\mu -1 =v_3 (n_2)+2$. Therefore, we obtain
\[
2 (A_n +mr)\equiv 2A_n -(2n_1+3n_2)\equiv 3n_2 (1+2\zeta) \pmod{ (3^{\mu -1}) }.
\]
Since $v_{\mathfrak p} (1+2\zeta )=1$, (\ref{eq:A3-tame}) holds in the case of (\ref{eq:f}),(1). 
In the case of (\ref{eq:f}),(2), we have $3|c$ and
\[
A_n +mr \equiv A_n \equiv 3t-6t (1+\zeta)\equiv -3t (1+2\zeta) \pmod{ (3^2)}.
\]
Since $v_{\mathfrak p} (1+2\zeta)=1,\, 3\nmid t$ and $3\nmid n_2$, (\ref{eq:A3-tame}) holds in the case of (\ref{eq:f}),(2).
Hence, we obtain (\ref{eq:delta3-tame}) on the $3$-adic valuation of $\delta$.

Next, we consider the $p$-adic valuation of $\delta$ for a prime number $p\, (\ne 2,3)$ dividing $n_2$.
Since $(n_2,ec)=1$, we have $p\nmid ec$; hence, $v_p (3ecn_2)=v_p(n_2)$.
Let $\mathfrak p$ be a prime idela of $\Z [\zeta]$ above $p$. We shall show
\begin{equation}\label{eq:deltan2-tame}
v_{\mathfrak p} (\delta)=v_{\mathfrak p}(n_2).
\end{equation}
The equality follows from
\begin{equation}\label{eq:An2-tame}
A_n +mr \equiv 0 \pmod{ (p^{v_p (n_2)}) } .
\end{equation}
We shall show the congruence. In the case of (\ref{eq:f}),(1), according to Proposition~\ref{prop:IB} and $p^{v_p (n_2)} |Q=
|n_2|/ (2^{v_2 (n_2)} 3^{v_3 (n_2)} )$, we have
\[
r\equiv -b (2\cdot 3^{\mu -1} \cdot a )^{\psi (Q)} 3^{\mu } \equiv -3b (2a)^{\psi (Q) } \pmod{p^{v_p (n_2)} }.
\]
From (\ref{eq:a,b}), we have
\[
c^2 (A_n +mr) \equiv c^2 (n_1+3n_2 (1+\zeta)-3bc (2a)^{\psi (Q) } ) \equiv 0 \pmod{ (p^{v_p (n_2)})}.
\]
Since $p\nmid c$, (\ref{eq:An2-tame}) holds in the case of  (\ref{eq:f}), (1).
In the case of (\ref{eq:f}),(2), according to Proposition~\ref{prop:IB} and $p^{v_p(n_2)} |Q= |n_2|/ (2^{v_2(n_2)}3^{v_3(n_2)}) $,
we have $r\equiv -3b (2a)^{\psi (Q) } \pmod{p^{v_p (n_2)} }$.
From (\ref{eq:a,b}), we have
\[
A_n+mr \equiv n_1+3n_2 (1+\zeta)-9bc (2a)^{\psi (Q)} \equiv 0 \pmod{ (p^{v_p(n_2)}) }.
\]
Therefore, (\ref{eq:An2-tame}) holds in the case of (\ref{eq:f}),(2).
Hence, we obtain (\ref{eq:deltan2-tame}) on 
the $p$-adic valuation of $\delta$ for a prime number $p\, (\ne 2,3)$ dividing $n_2$.

Finally, we consider the $p$-adic valuation of $\delta$ for a prime number $p\, (\ne 2,3)$ dividing $ec$.
Since $(n_2, ec)=1$, we have $p\nmid n_2$; hence, $v_p (3ecn_2)=v_p(ec)$. Since 
$p\equiv 1\pmod{3}$ from Lemma~\ref{lem:equiv-0,1}, we put $p=\pi \pi'$, where $\pi$ and $\pi'$ are prime elements of $\Z [\zeta]$
satisfying $\pi |A_n$ and $\pi' |A_n'$. Since $ec |\Delta_n =A_n A_n'$ and Lemma~\ref{lem:1-zeta}, we have $\pi^{v_p(ec)} |A_n$ and
$\pi' \nmid A_n$. 
According to (\ref{eq:a,b}), we have $e|a$ and $e|b$; hence, we have $e|r$ from Proposition~\ref{prop:IB}.
Put $\mathfrak p =(\pi ) $ and $\mathfrak p'=(\pi')$. We have $A_n +mr \equiv 0 \pmod{\mathfrak p^{v_p(ec)} }$ and
conclude
\begin{equation}\label{eq:deltaec-tame}
v_{\mathfrak p}(\delta )=v_p (ec), \quad v_{\mathfrak p'} (\delta)=0.
\end{equation}

%Let $p_1,\ldots, p_k$ be different prime numbers satisfying
%\begin{equation}\label{eq:p1k-tame}
%p_1^{\iota_1} \cdots p_k^{\iota_k} =\begin{cases}
%ec & \text{ in the case of (\ref{eq:f}),(1)}, \\
%e\cdot\frac{c}{3} & \text{ in the case of (\ref{eq:f}),(2)},
%\end{cases}
%\end{equation}
%and for $i\in \{1,\ldots, k \}$, we put $p_i =\pi_i \pi_i'$, where $\pi_i$ and $\pi_i'$ are prime elements
%of $\Z [\zeta]$ satisfying $ \pi_i |A_n$ and $\pi_i' |A_n'$. Since the integers $a_0$ and $a_1$ satisfy
%$ec=a_0^2-a_0a_1+a_1^2=(a_0+a_1\zeta)(a_0+a_1\zeta^2)$ and
%$a_0+a_1 \zeta $ divides $A_n$, we have the following equality uo tp units of $\Z [\zeta]$:
%\begin{equation}\label{eq:a0a1-tame}
%a_0 +a_1 \zeta =\begin{cases}
%\pi_1^{\iota_1} \cdots \pi_k^{\iota_k} & \text{ in the case of (\ref{eq:f}),(1)}, \\
%(1-\zeta) \pi_1^{\iota_1} \cdots \pi_k^{\iota_k}  & \text{ in the case of (\ref{eq:f}),(2)}. 
%\end{cases}
%\end{equation}
Combining (\ref{eq:nuI-tame}), (\ref{eq:p1k-tame}),  (\ref{eq:a0a1-tame}), (\ref{eq:delta-tame}), (\ref{eq:delta2-tame}), (\ref{eq:delta3-tame}), (\ref{eq:deltan2-tame})
 and (\ref{eq:deltaec-tame}), we obtain (\ref{eq:nuI-gen-tame}).

Put $x=9n_1n_2 (a_0 +a_1 \sigma ) \, (\in \Z [G])$. We have $\nu (x)=9n_1 n_2(a_0+a_1 \zeta)$. From (\ref{eq:nuI-tame}) and 
(\ref{eq:nuI-gen-tame}),
we obtain $\nu (g)=9n_1n_2 (a_0+a_1 \zeta)v$ for some $v\in \Z [\zeta ]^{\times}$.
Since $\Z [\zeta]^{\times} =\{ \pm 1, \pm \zeta, \pm \zeta^2 \}$ and
$\Z [G]^{\times} =\{ \pm 1, \pm \sigma , \pm \sigma^2 \}$, there exists $\xi \in \Z [G]^{\times}$ that satisfies
$\nu (\xi)=v^{-1}$, and we have $\nu (\xi g)=9n_1n_2 (a_0 +a_1 \xi)=\nu (x)$.
We conclude that $\xi g-x \in \Ker (\nu )=(1+\sigma +\sigma^2)$.
Therefore,
there exists $k\in \Z$ that satisfies 
\begin{equation}\label{eq:Tr-tame}
(\xi g -x)  \frac{\rho}{\ell} =k \Tr_{L_n/\Q} \left( \frac{\rho}{\ell} \right),
\end{equation}
where $\ell=9ec^2n_1$. From (\ref{eq:Tr-tame}), we have
\begin{align}
( \xi  g)  \dfrac{\rho}{\ell}  & =x  \dfrac{\rho}{\ell} +k \Tr_{L_n/\Q} \left( \dfrac{\rho}{\ell} \right) \label{eq:rho/ell-tame} \\
& = \dfrac{1}{ec^2} \left( n_2 (a_0\rho +a_1 \rho' )+\dfrac{k}{9n_2} \right). \notag
\end{align}
Since $de$ is square-free and $L_n \cap \Q (\zeta_{de} )=L_n$, we obtain from (\ref{eq:eta-tame})
\begin{align*}
\Tr_{L_n/\Q} \left((\xi g) \dfrac{\rho}{\ell} \right) & = \xi  \Tr_{L_n/\Q} \left( g \dfrac{\rho}{\ell } \right) =
\xi \Tr_{L_n/\Q}(\eta_{de}) =\xi  \Tr_{\Q (\zeta_{de})/\Q} (\zeta_{de}) = \xi  \mu (de)=\pm 1,
\end{align*}
where $\mu$ is the M\"{o}bius function.
Taking the trace on both sides of (\ref{eq:rho/ell-tame}) and multiplying them by $3n_2ec^2$, it follows that
\begin{equation}\label{eq:k-tame}
k=3n_2 (\varepsilon ec^2 -n_1(a_0+a_1)) 
\end{equation}
for $\varepsilon \in \{ \pm 1 \}$. Next, we determine the sign $\varepsilon \in  \{ \pm 1\}$.
From (\ref{eq:rho/ell-tame}), (\ref{eq:eta-tame}), and Lemma~\ref{lem:n2rho}, we have
\[
\frac{k}{9n_2} =ec^2 (\xi g)  \frac{\rho}{\ell} -n_2 (a_0 \rho+a_1 \rho' )\quad \in \calO_{L_n} \cap \Q =\Z,
\]
and hence, $9n_2$ divides $k$. Put $k'=k/(9n_2) \in \Z$.
From (\ref{eq:k-tame}), we obtain 
\begin{equation}\label{eq:k'-tame}
3k' =\varepsilon ec^2 -n_1 (a_0+a_1).
\end{equation}

First, we consider the case of (\ref{eq:f}), (1). In this case we have
$e\equiv c\equiv 1 \pmod{3}$;  hence, (\ref{eq:k'-tame}) yields $\varepsilon \equiv n_1 (a_0 +a_1) \pmod{3}$,
which implies $\varepsilon =\left( \frac{n_1 (a_0+a_1)}{3} \right)$.
We conclude from (\ref{eq:k-tame}) that 
\[
\frac{k}{9n_2}=\frac{1}{3}\left( \left( \frac{n_1(a_0+a_1)}{3} \right) ec^2 -n_1(a_0+a_1) \right) \quad
(\in \Z).
\]
Put $\alpha =(\xi g) (\rho/\ell)$. Then, we have $\calO_{L_n} =\Z [G]  \eta_{de} =
\Z [G]  \alpha$ from (\ref{eq:eta=g-tame}) and the assertion of the theorem in this case holds.

Next, we consider the case of (\ref{eq:f}), (2). First of all, we shall show 
\begin{equation}\label{eq:a0-a1-tame}
a_0 \equiv -a_1 \not\equiv 0 \pmod{3}.
\end{equation}
Since $ec=a_0^2-a_0a_1+a_1^2 \equiv (a_0+a_1)^2 \pmod{3}$ and $3||c$, we obtain $a_0+a_1 \equiv 0 \pmod{3}$.
If $3 |a_0$, then we have $3|a_1$, and this is a contradiction to $ec=a_0^2-a_0a_1+a_1^2$ and $3||c$. Therefore, we obtain (\ref{eq:a0-a1-tame}).
From (\ref{eq:k'-tame}), (\ref{eq:a0-a1-tame}), and  $3 |n_1$ and $3 ||c$, we have $3|k'$.
Put $\alpha = (\xi g)  (\rho/\ell)$. From (\ref{eq:rho/ell-tame}) and $k' =k/(9n_2)$,
using \cite[Lemma~4.2]{HA1}, we obtain
\[
N_{L_n/\Q}(\alpha)=\alpha \alpha' \alpha'' =\frac{1}{e^3c^6} R,
\]
where
\begin{align}
R & :=(a_0+a_1){k'}^2 n_1-a_0a_1^2 n_2(n_1^2 +3n_1n_2+9n_2^2)+a_0a_1 k' n_1^2 \label{eq:R-tame} \\
& \quad -(a_0^2-a_0a_1+a_1^2)k'n_2(n_1+3n_2)+{k'}^3 +(a_0+a_1)^3 n_2^3. \notag
\end{align}
Since $\alpha \in \calO_{L_n}$, we have $N_{L_n/\Q}(\alpha)=\alpha \alpha' \alpha'' \in \Z$.
Furthermore, since $3||c$, we have $3^6 |R$.
From (\ref{eq:a0-a1-tame}),  (\ref{eq:R-tame}), $3|n_1,\, 3|k',\, \Delta_n =n_1^2+3n_1n_2+9n_2^2 =de^2c^3$, and
$ec=a_0^2-a_0a_1 +a_1^2$, we have
\[
R  \equiv -a_0a_1^2 n_2 de^2c^3 +a_0a_1 k' n_1^2-ec k' n_2 (n_1+3n_2)+{k'}^3 +(a_0+a_1)^3 n_2^3 \equiv 0 \pmod{3^4}.
\]
Put $C=c/3,\, K=k'/3\, (\in \Z)$. Dividing the above congruence by $3^3$, we obtain
\begin{equation}
-a_0a_1^2 n_2 de^2 C^3 +a_0a_1 Kt^2-eCKn_2(t+n_2)+K^3+ \left( \frac{a_0+a_1}{3} \right)^3 n_2^3 \equiv 0 \pmod{3}.
\end{equation}
Since $d\equiv e\equiv C \equiv 1\pmod{3},\, t\equiv n_2 \pmod{3}$, and $a_1^2 \equiv 1 \pmod{3}$ from (\ref{eq:a0-a1-tame}),
this congruence yields
\begin{equation}\label{eq:mod3}
-a_0 n_2+a_0a_1 Kn_2^2 -2n_2^2 K +K^3+\left(\frac{a_0+a_1}{3} \right)^3 n_2^3 \equiv 0 \pmod{3}.
\end{equation}
On the other hand, we have $K=\varepsilon eC^2- (a_0+a_1)t/3$ from (\ref{eq:k'-tame}); hence, $(a_0+a_1)n_2/3 \equiv \varepsilon -K \pmod{3}$.
From this congruence and (\ref{eq:mod3}), we have
\[
-a_0 n_2+a_0a_1 Kn_2^2 -2n_2^2 K+\varepsilon \equiv 0 \pmod{3}.
\]
Furthermore, since $a_0a_1 \equiv -1 \pmod{3}$ from (\ref{eq:a0-a1-tame}), we obtain 
$\varepsilon \equiv n_2 a_0 \pmod{3}$, which implies $\varepsilon =\left( \frac{n_2 a_0}{3} \right)$.
We conclude from (\ref{eq:k-tame}) that $k/(9n_2)= \left( \left( \frac{n_2a_0}{3} \right) ec^2 -n_1 (a_0+a_1) \right) /3 \, (\in \Z)$.
From (\ref{eq:eta=g-tame})  and (\ref{eq:rho/ell-tame}), the assertion of the theorem in this case holds.
\end{proof}

For the special cases, the generator $\alpha$ of the ring of integers is given more simply, as described below.
First, we assume that $\Delta_n=n_1^2+3n_1n_2+9n_2^2$ is square-free. Then, we have $3\nmid n_1$ and $e=c=1$.
We obtain $(1,0)$ as a pair $(a_0,a_1)$ satisfying $ec=a_0^2-a_0a_1+a_1^2$ and $a_0+a_1\zeta |A_n$. From $\varepsilon =
\left(\frac{n_1(a_0+a_1)}{3} \right)=\left(\frac{n_1}{3} \right)$, the generator $\alpha$ of Thoerem~\ref{theo:tame} is given by $\alpha=
n_2\rho_n + \left( \left( \frac{n_1}{3} \right)-n_1 \right)/3$.
We obtain the following corollary.
\begin{cor}\label{cor:tame1}
Let $n=n_1/n_2$ be a rational number where the integers $n_1$ and $n_2$ are coprime.
Suppose that the cubic polynomial $f_n(X)$ is irreducible over $\Q$  and $\Delta_n$ is square-free.
Then $\alpha:= n_2\rho_n +\left( \left( \frac{n_1}{3} \right)-n_1 \right)/3$ is a generator of a normal integral basis of the cyclic cubic field $L_n$, 
namely, we have $\calO_{L_n}=\Z[G]\, \alpha =\calA_{L_n/\Q}\, \alpha$.
\end{cor}

Next, we assume that $\Delta_n=n_1^2+3n_1n_2+9n_2^2=3^3d$ with a square-free integer $d$.  Then, we have $3| n_1$ and $e=1,\, c=3$.
Furthermore, we assume $n_1=3t\, (t\in \Z)$ and $n_2 \equiv -2t \pmod{9}$. 
We obtain $(1,-1)$ as a pair $(a_0,a_1)$ satisfying $ec=a_0^2-a_0a_1+a_1^2$ and $a_0+a_1\zeta |A_n$. From $\varepsilon =
\left(\frac{n_2a_0}{3} \right)=\left(\frac{n_2}{3} \right)$, the generator $\alpha$ of Thoerem~\ref{theo:tame} is given by $\alpha=
 \left( n_2 (\rho_n-\rho'_n)+3\left( \frac{n_1}{3} \right) \right)/9$.
We obtain the following corollary.
\begin{cor}\label{cor:tame2}
Let $n=n_1/n_2$ be a rational number where the integers $n_1$ and $n_2$ are coprime.
Suppose that the cubic polynomial $f_n(X)$ is irreducible over $\Q$  and $\Delta_n=3^3d$ with a square-free integer $d$,
and $n_1=3t\, (t\in \Z),\, n_2\equiv -2t \pmod{9}$.
Then, $\alpha:= \left( n_2 (\rho_n-\rho'_n)+3\left( \frac{n_1}{3} \right) \right)/9  $ is a generator of a normal integral basis of the cyclic cubic field $L_n$, 
namely, we have $\calO_{L_n}=\Z[G] \, \alpha=\calA_{L_n/\Q} \, \alpha$.
\end{cor}
\section{Wildly ramified cases}\label{sec:wild}

Let $n_1/n_2$ be a rational number where the integers $n_1$ and $n_2$ are coprime.
Suppose that the cubic polynomial $f_n(X)$ is irreducible over $\Q$ and consider the case (3)
of (\ref{eq:f}). 
Let ${\bf 1}$ be the trivial character of $G$ and $\chi$ a character given by $\chi : G \to \C^{\times}, \ \sigma \mapsto \zeta_3$.
In this case, $L_n/\Q$ is wildly ramified since $3^2$ divides the conductor of $L_n$. The corresponding group of
Dirichlet character is
\[
\frakX =\begin{cases}
\Phi_{9de} \coprod \Phi_{3de}  & \text{ if $t\equiv n_2 \!\! \pmod{3}$}, \\
\Phi_{3de} \coprod \Phi_{de}  & \text{ if $t\not\equiv n_2 \!\! \pmod{3}$}, 
\end{cases}
\]
where $\Phi_{9de}= \{ \chi, \chi^2 \}$ and $\Phi_{3de}= \{ {\bf 1} \}$ if $t\equiv n_2 \pmod{3}$, and $
\Phi_{3de} = \{ \chi, \chi^2 \}$ and $\Phi_{de}= \{ {\bf 1} \}$ if $t\not\equiv n_2 \pmod{3}$.
%Put $\widetilde{\bf e}:={\bf e}_{\chi}+{\bf e}_{\chi^2}$.
Since 
\begin{align}
{\bf e}_{\mathfrak f_{L_n}} & = {\bf e}_{\chi}+{\bf e}_{\chi^2} =\dfrac{1}{3} (2-\sigma-\sigma^2) =
\begin{cases}
{\bf e}_{9de}  & \text{ if $t\equiv n_2 \!\! \pmod{3}$}, \\
{\bf e}_{3de}  & \text{ if $t\not\equiv n_2 \!\! \pmod{3}$}, 
\end{cases} \label{eq:e} \\
{\bf e}_{\bf 1} &=\dfrac{1}{3}  (1+\sigma +\sigma^2)  =
\begin{cases}
{\bf e}_{3de}  & \text{ if $t\equiv n_2 \!\! \pmod{3}$}, \\
{\bf e}_{de}  & \text{ if $t\not\equiv n_2 \!\! \pmod{3}$}, 
\end{cases} \notag
\end{align}
we have $\calA_{L_n/\Q}=\Z [G] [{\bf e}_{\mathfrak f_{L_n}}, {\bf e}_{\bf 1} ]$ and
\begin{equation}\label{eq:int-ring-wild}
\calO_{L_n} = \begin{cases}
\Z [G] \, \eta_{9de} \oplus \Z & \text{ if $t\equiv n_2 \!\! \pmod{3}$}, \\
\Z [G] \,  \eta_{3de} \oplus \Z  & \text{ if $t\not\equiv n_2 \!\! \pmod{3}$}.
\end{cases}
\end{equation}
Put $\rho = \rho_n$, $\zeta = \zeta_3$ and  $\Delta_n =n_1^2+3n_1n_2+9n_2^2=A_n A_n'$ with
$A_n:=n_1+3n_2(1+\zeta),\, A_n':=n_1+3n_2(1+\zeta^2)$.
In this case, the generators of the $\Z [G]$-module $\calO_{L_n}$ are obtained by the
roots of $f_n(X)$ as follows.

\begin{theo}\label{theo:wild}
Let $n=n_1/n_2$ be a rational number where the integers $n_1$ and $n_2$ are coprime.
Suppose that the cubic polynomial $f_n(X)$ is irreducible over $\Q$, and $n_1=3t\, (t\in \Z)$, $n_2 \not\equiv -2t \pmod{9}$. There exist integers $a_0$ and $a_1 $  that satisfy
\[ec=
\begin{cases}
a_0^2-a_0a_1+a_1^2 &  \text{ if $t\equiv n_2 \!\! \pmod{3}$}, \\
3(a_0^2-a_0a_1+a_1^2)  &   \text{ if $t\not\equiv n_2 \!\! \pmod{3}$},
\end{cases}
\]
and
$a_0+a_1 \zeta $ divides $A_n$ in $\Z [\zeta]$. Furthermore, for any such $a_0,\, a_1$ and 
\[
\alpha :=\frac{1}{ec^2}(3n_2 (a_0\rho_n +a_1 \rho'_n)-n_1(a_0+a_1)),
\]
we have  $\alpha \in {\bf e}_{\mathfrak f_{L_n}}\calO_{L_n}$ and 
\[
\calO_{L_n}=\Z [G] \,  \alpha \oplus \Z =\calA_{L_n/\Q}  (\alpha+1),
\]
where ${\bf e}_{\mathfrak f_{L_n}}=(2-\sigma-\sigma^2)/3$.
\end{theo}
\begin{rem}\label{rem:wild}
The result in the case of the simplest cubic fields (namely, when $n_2=\pm 1$) has been obtained
in \cite{OA}.
\end{rem}
\begin{proof} 
When $n=0$ ($n_1=0$ and $n_2=\pm 1$), it is the case of $t\not\equiv n_2 \pmod{3}$.
In this case, we have $f_n(X)=X^3-3X-1$ (irreducible over $\Q$), $\Delta_n=3^2,\, d=1,\, e=3,\, c=1$, and
$D_{L_n}=3^4$.
A pair of integers $a_0$ and $a_1$ satisfying $3(a_0^2-a_0a_1+a_1^2)=ec$ and $a_0+a_1 \zeta | A_n$ in $\Z [\zeta]$
is $a_0=1, a_1=0$. Then, we have $\alpha =\pm \rho$.
In fact, since $\{ 1,\rho, \rho' \}$ is an integral basis, the assertion $\calO_{L_n} =\Z [G]  \rho \oplus \Z$
of the theorem holds.
Hereafter, we assume $n\ne 0$. Let $\{ 1, \phi, \psi \}$ be the integral basis of $L_n$ given by Proposition~\ref{prop:IB}.
We find $\alpha$ satisfying $\calO_{L_n} =\Z [G] \alpha \oplus \Z$ using the method proposed by Acciaro and 
Fieker \cite{AF}. From $\rho +\rho' +\rho'' =n$ and $\rho^2=\rho'+(n+1) \rho +2$, we obtain the following:
\begin{align*}
1 & =\dfrac{n_2}{n_1}(\rho+ \rho' +\rho'' ),\\
\phi & =\dfrac{n_2}{m} (3\rho- (\rho+\rho'+\rho'')) ,\\
\psi & =\dfrac{1}{3ec^2n_1} (9 (n_1n_2 \rho' +n_1 (n_1+n_2)\rho +2n_2^2(\rho+\rho'+\rho'' ))  \\
&  \quad +3n_1 (mr -2n_1)\rho +(m^2r^2+m^2 a-mn_1 r+n_1^2)(\rho+\rho' +\rho'' )),
\end{align*}
where
\[
m = 
\begin{cases}
c &  \text{ if $t\equiv n_2 \!\! \pmod{3}$}, \\
3c &   \text{ if $t\not\equiv n_2 \!\! \pmod{3}$}.
\end{cases}
\]
Put $\ell =3ec^2$. We have
\[
{\bf e}_{\mathfrak f_{L_n}}  1=g_1 \, \frac{\rho}{\ell}, \quad {\bf e}_{\mathfrak f_{L_n}}  \phi =g_2\,  \frac{\rho}{\ell}, \quad {\bf e}_{\mathfrak f_{L_n}}  \psi =g_3\,  \frac{\rho}{\ell},
\]
 where $g_1,g_2,g_3 \in \Z [G]$ are given by
\begin{align}
g_1 & =0 ,\notag\\
g_2 & =\dfrac{\ell n_2} {m} (2-\sigma -\sigma^2), \label{eq:g-wild} \\
g_3 & =\dfrac{\ell}{3ec^2} ((2n_1+3n_2+2mr)-(n_1-3n_2+mr)\sigma -(n_1+6n_2+mr) \sigma^2). \notag
\end{align}
Since $\{ {\bf 1}, \phi, \psi \}$ is an integral basis, we obtain from (\ref{eq:g-wild}) that
\begin{equation}\label{eq:O-wild}
{\bf e}_{\mathfrak f_{L_n}}  \calO_{L_n}=({\bf e}_{\mathfrak f_{L_n}}  \phi )\Z +({\bf e}_{\mathfrak f_{L_n}}  \psi ) \Z =(g_2 \Z [G] +g_3 \Z [G] )
\frac{\rho}{\ell} .
\end{equation}
Since
\begin{equation}\label{eq:em}
{\bf e}_m \eta_{m'} = \begin{cases}
0 & \it{if}\ m\ne m', \\
\eta_m &  \it{if} \ m=m',\\
\end{cases}
\end{equation}
for any $m,m' \in \calD (\mathfrak f_{L_n})$ (see \cite[p.165, (2) and p.167, lines 1,2]{Le}), we have
from (\ref{eq:e}) and (\ref{eq:int-ring-wild})
\begin{equation}\label{eq:eO-wild}
{\bf e}_{\mathfrak f_{L_n}} \calO_{L_n} =\begin{cases}
\Z [G]  \eta_{9de}  & \text{ if $t\equiv n_2 \!\! \pmod{3}$}, \\
\Z [G]  \eta_{3de} & \text{ if $t\not\equiv n_2 \!\! \pmod{3}$}.
\end{cases}
\end{equation}
From (\ref{eq:O-wild}) and (\ref{eq:eO-wild}), there exists $g\in \Z [G]$ satisfying
\begin{equation}\label{eq:eta-wild}
g \, \frac{\rho}{\ell} = \begin{cases}
\eta_{9de}  & \text{ if $t\equiv n_2 \!\! \pmod{3}$}, \\
\eta_{3de} & \text{ if $t\not\equiv n_2 \!\! \pmod{3}$}.
\end{cases}
\end{equation}
Therefore, we have
\begin{equation}\label{eq:eta=g-wild}
{\bf e}_{\mathfrak f_{L_n}}  \calO_{L_n} =
g\Z [G]  \frac{\rho}{\ell} 
= \begin{cases}
\Z [G]  \eta_{9de}  & \text{ if $t\equiv n_2 \!\! \pmod{3}$}, \\
\Z [G]  \eta_{3de} & \text{ if $t\not\equiv n_2 \!\! \pmod{3}$}.
\end{cases}
\end{equation}
From (\ref{eq:O-wild}) ad (\ref{eq:eta=g-wild}), we obtain the equality of ideals of $\Z [G]$:
\[
(g) +\mathrm{Ann}_{\Z [G] }\left( \frac{\rho}{\ell} \right) =(g_2,g_3) +\mathrm{Ann}_{\Z [G] }\left( 
\frac{\rho}{\ell} \right).
\]
It follows from $n\ne 0$ that $\{ \rho/\ell, \rho'/\ell, \rho''/\ell \}$ is a normal
basis (\cite[Proof of Theorem~4.4]{HA1}).
Therefore, we have $\mathrm{Ann}_{\Z [G] }(\rho/\ell)=0$, and obtain the equality of idelas of $\Z [G]$:
\begin{equation}\label{eq:g12-wild}
(g)=(g_2,g_3).
\end{equation}
Consider the surjective ring homomorphism
\[
\nu: \Z [G] \longrightarrow \Z [\zeta],
\]
defined by $\nu (\sigma )=\zeta$. We calculate the image of the ideal $I:=(g)=(g_2,g_3)$ by $\nu$.
 Since $\nu$ is surjective, we obtain the ideal of $\Z [\zeta]$:
\begin{equation}\label{eq:nu23^wild}
\nu (I) =( \nu (g)) =(\nu (g_2), \nu(g_3)).
\end{equation}
 From (\ref{eq:g-wild}), the elements $\nu (g_2)$ and $\nu (g_3)$ are given by
 \begin{align*}
 \nu (g_2) & =\frac{\ell n_2}{m} (2-\zeta-\zeta^2)=\frac{9ec^2n_2}{m} \\
 \nu (g_3) & = \frac{\ell}{3ec^2} ( (2n_1+3n_2+2mr)-(n_1-3n_2+mr) \zeta -(n_1+6n_2+mr) \zeta^2 ) \\
 & =3(A_n+mr),
 \end{align*}
 where $A_n =n_1+3n_2(1+\zeta)$. Therefore, we obtain
 \begin{equation}\label{eq:nuI-wild}
 \nu (I)= (\nu (g))=3 \left( \frac{3ec^2n_2}{m}, A_n+mr \right).
 \end{equation}
 Let $p_2, \ldots ,p_k$ be different prime numbers satisfying
 \begin{equation}\label{eq:p1k-wild}
p_1^{\iota_1} \cdots p_k^{\iota_k} =\frac{ec}{3},
 \end{equation}
 with $\iota_1,\ldots, \iota_k \in \Z_{>0}$. For 
 $i\in  \{1,\ldots, k \}$, we put $p_i =\pi_i \pi_i'$, where $\pi_i$ and $\pi_i'$ are prime elements
 of $\Z [ \zeta]$ satisfying $\pi_i |A_n$ and $\pi' |A_n'$.
 Let $a_0$ and $a_1$ be integers satisfying
 \begin{equation}\label{eq:a0a1-wild}
a_0+a_1\zeta =  \begin{cases}
 (1-\zeta) \pi_1^{\iota_1} \cdots \pi_k^{\iota_k}  & \text{ if $t\equiv n_2 \!\! \pmod{3}$}, \\
\pi_1^{\iota_1} \cdots \pi_k^{\iota_k} & \text{ if $t\not\equiv n_2 \!\! \pmod{3}$}.
\end{cases}
 \end{equation}
 Then $a_0$ and $a_1$ satisfy
 \begin{align}
& ec = \begin{cases}
a_0^2-a_0a_1+a_1^2 =(a_0+a_1 \zeta)(a_0+a_1\zeta^2)  & \text{ if $t\equiv n_2 \!\! \pmod{3}$}, \\
3(a_0^2-a_0a_1+a_1^2) =3(a_0+a_1\zeta)(a_0 +a_1\zeta^2) & \text{ if $t\not\equiv n_2 \!\! \pmod{3}$}, 
 \end{cases}  \label{eq:ec} \\
& a_0+a_1 \zeta \, | \, A_n. \notag
\end{align}
Conversely, any integers $a_0$ and $a_1$ satisfying (\ref{eq:ec}), $a_0+a_1\zeta$ is given by the right-hand side of (\ref{eq:a0a1-wild}) up to a unit of
$\Z [\zeta]$.

 Next, we shall prove
 \begin{equation}\label{eq:nuI-gen-wild}
 \nu (I)=9n_2 (a_0 +a_1 \zeta).
 \end{equation}
% where $a_0$ and $a_1$ are integers that satisfy
 %\[
 %ec= \begin{cases}
 %a_0^2-a_0a_1+a_1^3 & \text{ if $t\equiv n_2 \!\! \pmod{3}$}, \\
 %3(a_0^2-a_0a_1+a_1^2) & \text{ if $t\not\equiv n_2 \!\! \pmod{3}$},
% \end{cases}
% \]
% and $a_0+a_1 \zeta$ divides $A_n$ in $\Z [\zeta]$.
 Put
 \begin{equation}\label{eq:delta-wild}
 \left( \frac{3ec^2n_2}{m}, A_n+mr \right)=(\delta), \quad \delta \in \Z [\zeta]
 \end{equation}
 and consider the $\mathfrak p$-adic valuation of $\delta$ for prime ideals $\mathfrak p$
 of $\Z [\zeta]$. Obviously we have $v_{\mathfrak p}(\delta)=0$ if $v_{\mathfrak p} (3ecn_2)=0$.
 
 First, we consider the $2$-adic valuation of $\delta$. Let $\mathfrak p=(2)$ be the prime idela of 
 $\Z [\zeta]$ above $2$. We shall show 
 \begin{equation}\label{eq:delta2-wild}
 v_{\mathfrak p}(\delta )=v_2(n_2).
 \end{equation}
 We have $2\nmid ec$ from Lemma~\ref{lem:equiv-0,1}, and hence $v_2 (3ec^2n_2/m)=v_2(n_2)$.
 Therefore, (\ref{eq:delta2-wild}) follows from
 \begin{equation}\label{eq:A2-wild}
 A_n +mr \equiv 0 \pmod{ (2^{v_2(n_2)})}.
 \end{equation}
 We shall show this congruence. If $2\nmid n_2$, then (\ref{eq:A2-wild}) holds.
 We consider the case $2||n_2$.
 From Lemmas ~\ref{lem:equiv-0,1}, \ref{lem:lambda-mu}, and (\ref{eq:a,b}), we have $2\nmid a,\, 2|b$ and
 $2\nmid Q$.
 Furthermore, we have $2\nmid m$ from Lemmas~\ref{lem:equiv-0,1} and \ref{lem:m}. According to Proposition~\ref{prop:IB}, we have
 $r\equiv 3Q^2a \equiv 1 \pmod{2}$. Combining all these, we have
 $A_n+mr \equiv 0 \pmod{(2)}$. Hence, (\ref{eq:A2-wild}) holds in the case of
 $2|| n_2$. Next, we consider the case of $4|n_2$.
 In both cases (3) (i) and (3) (ii) of (\ref{eq:f}), according to Proposition~\ref{prop:IB} and Lemma~\ref{lem:lambda-mu},
 we have $r\equiv (2^{\lambda -1}-3\cdot b/2)a^{\psi (2^{\lambda})} \pmod{2^{\lambda}}$ and $\lambda =v_2(n_2)$.
 In the case of (\ref{eq:f}), (3) (i), we have $m=c$, and from (\ref{eq:a,b}), we have
 \begin{align*}
 a(A_n+mr) & \equiv a(n_1 +3n_2 (1+\zeta)) +ac (2^{\lambda -1}-3\cdot b/2 )a^{\psi (2^{\lambda})} \\
 & \equiv a n_1 +c (2^{\lambda -1} -3\cdot b/2) \\
 & \equiv -3de^2 c n_1+c \left( 2^{\lambda-1 }+3de^2 \cdot \dfrac{2n_1+3n_2}{2} \right) \\
 & \equiv 2^{\lambda -1} c \left( 1+9de^2 \dfrac{n_2}{2^{\lambda} }\right) \\
 & \equiv 0 \pmod{ (2^{\lambda})} .
 \end{align*}
 Since $2\nmid a$, (\ref{eq:A2-wild}) holds in the case of $4|n_2$ and (\ref{eq:f}), (3) (i).
 In the case of $4| n_2$ and (\ref{eq:f}), (3) (ii), we have $m=3c$, and from (\ref{eq:a,b}), we have
 \begin{align*}
 a (A_n +mr) & \equiv a(n_1+3n_2 (a+\zeta)) +3ac (2^{\lambda -1} -3\cdot b/2) a^{\psi (2^{\lambda})} \\
 & \equiv an_1 +3c (2^{\lambda -1} -3\cdot b/2 ) \\
 & \equiv -d\dfrac{e^2}{3} c n_1 +3c \left (2^{\lambda -1} +d \left( \dfrac{e}{3} \right)^2 \dfrac{2n_1+3n_2}{2} \right) \\
 & \equiv 2^{\lambda -1}  c \left( 3+ de^2 \dfrac{n_2}{2^{\lambda} } \right) \\
 & \equiv 0 \pmod{ (2^{\lambda})}.
 \end{align*}
 Since  $2\nmid a$, (\ref{eq:A2-wild}) holds in the case of $4|n_2$ and (\ref{eq:f}), (3) (ii). Hence, we obtain (\ref{eq:delta2-wild})
 on the $2$-adic valuation of $\delta$.

 Next, we consider the $3$-adic valuation of $\delta$. Let $\mathfrak p =(1-\zeta)$ be the prime ideal of $\Z [\zeta]$ above $3$.
 We shall show
 \begin{equation}\label{eq:delta3-wild}
 v_{\mathfrak p} (\delta) =\begin{cases}
 3 & \text{ if $t\equiv n_2 \!\! \pmod{3}$}, \\
2 & \text{ if $t\not\equiv n_2 \!\! \pmod{3}$},
 \end{cases}
 \end{equation}
 Note that $3\nmid n_2$ since $3|n_1$. From Lemma~\ref{lem:m}, we have
 \begin{equation}\label{eq:val3-wild}
 v_{\mathfrak p} \left( \frac{3ec^2n_2}{m} \right) =2v_3 \left( \frac{3ec^2n_2}{m} \right)=
 \begin{cases}
 4 & \text{ if $t\equiv n_2 \!\! \pmod{3}$}, \\
2 & \text{ if $t\not\equiv n_2 \!\! \pmod{3}$}.
 \end{cases}
 \end{equation}
 In the case of $t\equiv n_2 \pmod{3}$, according to Proposition~\ref{prop:IB}, Lemma~\ref{lem:m} and 
 (\ref{eq:a,b}), we have $3^2 |a,\, 3^2 |b$ and hence $3^2 |r$. Since $3|m=c$, we obtain $3^3 |mr$.
 On the other hand, we have $v_{\mathfrak p} (A_n)=3$ since
 \begin{align*}
 \dfrac{1}{3}  A_n & =t+n_2 (1+\zeta) \equiv 3t \equiv 0 \pmod{\mathfrak p} ,\\
 \dfrac{1}{3} A_n & = (t+n_2)+n_2 \zeta \not\equiv 0 \pmod{ (3)}.
 \end{align*}
 Therefore, we have $v_{\mathfrak p} (A_n+mr)=3$. From (\ref{eq:val3-wild}), we obtain (\ref{eq:delta3-wild}).
 In the case of $t\not\equiv n_2 \pmod{3}$, according to Lemmas~\ref{lem:equiv-0,1}, 
 \ref{lem:m},
 \ref{lem:lambda-mu}, Proposition~\ref{prop:IB},
 and (\ref{eq:a,b}), we have 
 \[
 r\equiv -bQ^2 \equiv -b \equiv d \left( \frac{e}{3} \right)^2 (2t+n_2) \equiv 2t +n_2 \pmod{3}.
 \]
 Therefore, we have
 \[
 \frac{1}{3} (A_n +mr) =t +n_2 (1+\zeta) +cr \equiv 3(t+n_2) \equiv 0 \pmod{\mathfrak p},
 \]
 and hence $v_{\mathfrak p} (A_n+mr) \geq 3$. From (\ref{eq:val3-wild}), we obtain 
 (\ref{eq:delta3-wild}).
 
 Next, we consider the $p$-adic valuation of $\delta$ for a prime number $p\, (\ne 2,3)$ dividing $n_2$.
 Since $(n_2,ec)=1$, we have $p\nmid ec$, and hence $v_p \left(\frac{3ec^2n_2}{m} \right)=v_p (n_2)$.
 Let $\mathfrak p$ be a prime ideal of $\Z [\zeta]$ above $p$.
 We shall show
 \begin{equation}\label{eq:deltan2-wild}
 v_{\mathfrak p} (\delta )=v_{\mathfrak p} (n_2).
 \end{equation}
 The equality follows from
 \begin{equation}\label{eq:An2-1-wild}
 A_n+mr \equiv 0 \pmod{ (p^{v_p (n_2)})} .
 \end{equation}
 We shall show the congruence. According to Proposition~\ref{prop:IB} and $p^{v_p (n_2)} | Q=|n_2 |/
 (2^{v_2(n_2)} 3^{v_3(n_2)} )$, we have
 $r\equiv -3b (2a)^{\psi (Q) } \pmod{p^{v_p (n_2)} }$, and hence
 \begin{equation}\label{eq:An2-2-wild}
 A_n+mr \equiv n_1-3mb (2a)^{\psi (Q)} \pmod{ (p^{v_p (n_2)} )}.
 \end{equation}
 From (\ref{eq:a,b}) and (\ref{eq:An2-2-wild}), we obtain (\ref{eq:An2-1-wild}).
 Hence, we obtain (\ref{eq:deltan2-wild}) on the $p$-adic valuation of $\delta$ for a prime number $p\, (\ne 2,3)$ dividing $n_2$.

 Finally, we consider the $p$-adic valuation of $\delta$ for a prime number $p\, (\ne 2,3)$ dividing $ec$.
 Since $(n_2, ec)=1$, we have $p\nmid n_2$ and hence $v_p \left( \frac{3ec^2n_2}{m} \right)=v_p (ec)$.
 Since $p\equiv 1 \pmod{3}$ from Lemma~\ref{lem:equiv-0,1}, we put $p=\pi \pi'$ where $\pi$ and $\pi'$ are
 prime elements of $\Z [\zeta]$ satisfying $\pi |A_n$ and $\pi' |A_n'$. Since $ec |\Delta_n =A_nA_n'$ and
 Lemma~\ref{lem:1-zeta}, we have $\pi^{v_p (ec)} |A_n$ and $\pi' \nmid A_n$.
 According to (\ref{eq:a,b}), we have $e|a$ and $e/3^{v_3(e)} | b$, and hence we have
 $e/3^{v_3 (e)} |r$ from Proposition~\ref{prop:IB}.
 We obtain $mr \equiv 0 \pmod{ec}$.
 Put $\mathfrak p: =(\pi )$ and $\mathfrak p':=(\pi')$.
 We have $A_n+mr \equiv 0 \pmod{\mathfrak p^{v_p(ec)}}$ and conclude
 \begin{equation}\label{eq:deltaec-wild}
 v_{\mathfrak p} (\delta) =v_p (ec), \quad v_{\mathfrak p'} (\delta)=0.
 \end{equation}

 %Let $p_2, \ldots ,p_k$ be different prime numbers satisfying
 %\begin{equation}\label{eq:p1k-wild}
% p_1^{\iota_1} \cdots p_k^{\iota_k} =\frac{ec}{3},
 %\end{equation}
 %and for $i\in  \{1,\ldots, k \}$, we put $p_i =\pi_i \pi_i'$, where $\pi_i$ and $\pi_i'$ are prime elements
 %of $\Z [ \zeta]$ satisfying $\pi_i |A_n$ and $\pi' |A_n'$.
 %Since the integers $a_0$ and $a_1 $ satisfy
% \begin{equation}\label{eq:ec}
% ec= \begin{cases}
%a_0^2-a_0a_1+a_1^2 =(a_0+a_1 \zeta)(a_0+a_1\zeta^2)  & \text{ if $t\equiv n_2 \!\! \pmod{3}$}, \\
%3(a_0^2-a_0a_1+a_1^2) =3(a_0+a_1\zeta)(a_0 +a_1\zeta^2) & \text{ if $t\not\equiv n_2 \!\! \pmod{3}$},
% \end{cases}
% \end{equation}
% and $a_0+a_1\zeta$ divides $A_n$, we have the following equality up to
% units of $\Z [\zeta]$:
% \begin{equation}\label{eq:a0a1-wild}
% a_0+a_1\zeta =  \begin{cases}
% (1-\zeta) \pi_1^{\iota_1} \cdots \pi_k^{\iota_k}  & \text{ if $t\equiv n_2 \!\! \pmod{3}$}, \\
%\pi_1^{\iota_1} \cdots \pi_k^{\iota_k} & \text{ if $t\not\equiv n_2 \!\! \pmod{3}$}.
% \end{cases}
 %\end{equation}
 Combining (\ref{eq:nuI-wild}),  (\ref{eq:p1k-wild}),  (\ref{eq:a0a1-wild}), (\ref{eq:ec}), (\ref{eq:delta-wild}),  
 (\ref{eq:delta2-wild}), (\ref{eq:delta3-wild}), (\ref{eq:deltan2-wild}) 
 and (\ref{eq:deltaec-wild}), we obtain (\ref{eq:nuI-gen-wild}).

 Put $x=9n_2 (a_0+a_1 \sigma )\, (\in \Z [G])$.
 Similarly to the proof of Theorem~\ref{theo:tame}, there exist $\xi \in \Z [G]^{\times}$ and $k\in \Z $ that satisfy
\begin{align}
(\xi g)  \dfrac{\rho}{\ell} & =x  \dfrac{\rho}{\ell} +k \Tr_{L_n/\Q} \left(\dfrac{\rho}{\ell} \right) \label{eq:rho/ell-wild} \\
 & =\dfrac{1}{3ec^2n_2} (9n_2^2 (a_0\rho +a_1 \rho' ) +kn_1). \notag
\end{align}
Since $9de$ (resp. $3de$) is not square-free in the case of $t\equiv n_2 \pmod{3}$  (resp. $t\not\equiv n_2 \pmod{3}$),
we obtain from (\ref{eq:eta-wild})
\begin{align*}
\Tr_{L_n/\Q} \left( (\xi g)  \dfrac{\rho}{\ell} \right) & =\xi  \Tr_{L_n/\Q}\left(g \dfrac{\rho}{\ell} \right) \\
& =\begin{cases}
 \xi  \Tr_{L_n/\Q} (\eta_{9de} )=\xi  \Tr_{\Q(\zeta_{9de})/\Q} (\zeta_{9de})=\xi  \mu (9de) & \text{ if $t\equiv n_2 \!\! \pmod{3}$}, \\
 \xi  \Tr_{L_n/\Q} (\eta_{3de} )=\xi  \Tr_{\Q(\zeta_{3de})/\Q} (\zeta_{3de})=\xi  \mu (3de) & \text{ if $t\not\equiv n_2 \!\! \pmod{3}$},
 \end{cases}\\
 & =0,
\end{align*}
where $\mu$ is the  M\"{o}bius function. Taking the trace on both sides of (\ref{eq:rho/ell-wild}) and multiplying them by $ec^2n_2/n_1$,
it follows that $k=-3n_2 (a_0+a_1)$.
Put $\alpha := (\xi g) (\rho/\ell)$.  we have $\alpha \in {\bf e}_{\mathfrak f_{L_n}} \calO_{L_n}$ from (\ref{eq:eta=g-wild}). From $\calO_{L_n} ={\bf e}_{\mathfrak f_{L_n}}  \calO_{L_n} \oplus \Z$, (\ref{eq:eta=g-wild}) and (\ref{eq:rho/ell-wild}),
the assertion of the theorem holds.
\end{proof}
For the special cases, the generator $\alpha$ of the ring of integers is given more simply, as described below.
First, we assume that $\Delta_n=n_1^2+3n_1n_2+9n_2^2=3^3d$ with a square-free integer $d$. Then, we have $3| n_1$ and $e=1,\, c=3$.
Furthermore, we assume $n_1=3t\, (t\in \Z)$ and $n_2 \not\equiv -2t \pmod{9}$.
In this case, we have $t\equiv n_2 \pmod{3}$ from Lemma~\ref{lem:m}.
We obtain $(1,-1)$ as a pair $(a_0,a_1)$ satisfying $ec=a_0^2-a_0a_1+a_1^2$ and $a_0+a_1\zeta |A_n$. 
The generator $\alpha$ of Thoerem~\ref{theo:wild} is given by $\alpha=
n_2(\rho_n-\rho_n')/3$.
We obtain the following corollary.
\begin{cor}\label{cor:wild1}
Let $n=n_1/n_2$ be a rational number where the integers $n_1$ and $n_2$ are coprime.
Suppose that the cubic polynomial $f_n(X)$ is irreducible over $\Q$  and $\Delta_n=3^3d$ with a square-free integer $d$,
and $n_1=3t \, (t\in \Z),\, n_2\not\equiv -2t \pmod{9}$. For
$\alpha:= n_2(\rho_n-\rho_n')/3$, we have $\alpha \in {\bf e}_{\mathfrak f_{L_n}} \calO_{L_n}$ and $\calO_{L_n} =\Z [G] \, \alpha \oplus \,  \Z =\calA_{L_n/\Q} (\alpha +1)$.
\end{cor}

Next, we assume that $\Delta_n=n_1^2+3n_1n_2+9n_2^2=3^2d$ with a square-free integer $d$ and $3\nmid d$. Then, we have $3| n_1$ and $e=3,\, c=1$.
In this case, we have $n_1=3t\, (t\in \Z)$ and $t\not\equiv n_2 \pmod{3}$ from Lemma~\ref{lem:m}.
We obtain $(1,0)$ as a pair $(a_0,a_1)$ satisfying $ec=3(a_0^2-a_0a_1+a_1^2)$ and $a_0+a_1\zeta |A_n$. 
The generator $\alpha$ of Thoerem~\ref{theo:wild} is given by $\alpha=
(3n_2 \rho_n-n_1)/3$.
We obtain the following corollary.
\begin{cor}\label{cor:wild2}
Let $n=n_1/n_2$ be a rational number where the integers $n_1$ and $n_2$ are coprime.
Suppose that the cubic polynomial $f_n(X)$ is irreducible over $\Q$  and $\Delta_n=3^2d$ with a square-free integer $d$ and $3\nmid d$.
 For
$\alpha:=  (3n_2 \rho_n-n_1)/3     $, we have $\alpha \in {\bf e}_{\mathfrak f_{L_n}} \calO_{L_n}$ and $\calO_{L_n} =\Z [G] \, \alpha \oplus \Z =\calA_{L_n/\Q} (\alpha +1)$.
\end{cor}
 \section{Examples}\label{sec:ex}
 
 We give examples for $n_2=2,3,5,7,11$ and $-(n_2-1)\leq n_1 \leq n_2-1$ by using Magma.
  For these $(n_1,n_2)$,
 Tables~\ref{table:2}-\ref{table:11} give the discriminant $D_{L_n}$ for $n=n_1/n_2$, the corresponding case of (\ref{eq:f}), $\Delta_n$,
 $(a_0, a_1)$ and $\alpha$ in Theorem~\ref{theo:tame} (resp. Theorem~\ref{theo:wild}) in the case of (1) or (2) (resp. (3) (i) or (ii)), and Galois module structure of $\calO_{L_n}$.
 
 For example, for $(n_1,n_2)=(1,2)$ in Table~\ref{table:2}, the discriminant for $n=n_1/n_2=1/2$ is $D_{L_n}=43^2$ and Case (1) of (\ref{eq:f}) holds.
 Since $\Delta_n=43$, we have $d=43, e=1, c=1$, and hence we obtain $(1,0)$ as a pair $(a_0,a_1)$ satisfying $ec=a_0^2-a_0a_1+a_1^2$ and $a_0+a_1\zeta | A_n$.
 From $\varepsilon=\left(\frac{n_1(a_0+a_1)}{3} \right)=\left(\frac{1}{3} \right)=1$, in Theorem~\ref{theo:tame}, $\alpha=2\rho$.
From Theorem~\ref{theo:tame},  for this $\alpha$, the Galois module structure  of the ring of integers of $\calO_{L_n}$ is given by $\calO_{L_n} =\Z [G] \alpha$.

For $(n_1,n_2)=(6,11)$ in Table~\ref{table:11}, the discriminant for $n=n_1/n_2=6/11$ is $D_{L_n}=3^4 \cdot 7^2$ and Case (3) (i) of (\ref{eq:f}) holds.
 Since $\Delta_n=3^3 \cdot 7^2$, we have $d=1, e=7, c=3$, and hence we obtain $(5,1)$ as a pair $(a_0,a_1)$ satisfying $ec=a_0^2-a_0a_1+a_1^2$ and $a_0+a_1\zeta | A_n$.
 In Theorem~\ref{theo:wild}, $\alpha=(55/21) \rho+(11/21) \rho'-4/7$.
From Theorem~\ref{theo:wild},  for this $\alpha$, the Galois module structure  of the ring of integers of $\calO_{L_n}$ is given by $\calO_{L_n} =\Z [G] \alpha \oplus \Z$.

In the following table, using $(-n_1, -n_2)$ instead of $(n_1,n_2)$ yields $-\alpha$ as a generator.
 
 \begin{table}[h]
\caption{$n_2=2$}
%{\small 
 \begin{center}
 \begin{tabular}{ |  >{\centering} p{4.5em}|>{\centering} p{4em}|>{\centering} p{3em}|>{\centering} p{4em}|>{\centering} p{3em}|>{\centering} p{6em}|c|}
 \hline
\rule{0pt}{5mm}   $(n_1,n_2)$ & $D_{L_n}$ & Case & $\Delta_n$ & $(a_0,a_1)$ & $\alpha$ & \hspace{5mm} $\mathcal O_{L_n}$  \hspace{5mm} \\ 
\hline  \hline 
\rule{0pt}{5mm}   $(-1,2)$ & $31^2$ & (1) & $31$ & $(1,0)$ & $2\rho $ & $\Z [G] \alpha$ \\ 
\hline
\rule{0pt}{5mm}   $(1,2)$ & $43^2$ & (1) & $43$ & $(1,0)$ & $2\rho $ & $\Z [G] \alpha$ \\ 
\hline
\end{tabular}
\label{table:2}
\end{center}
%}
\end{table}
 \begin{table}[H]
\caption{$n_2=3$}
%{\small 
 \begin{center}
\begin{tabular}{| >{\centering}p{4.5em}|>{\centering} p{4em}|>{\centering} p{3em}|>{\centering} p{4em}|>{\centering} p{3em}|>{\centering} p{6em}| c|}
\hline
\rule{0pt}{5mm}   $(n_1,n_2)$ & $D_{L_n}$ & Case & $\Delta_n$ & $(a_0,a_1)$ & $\alpha$ & \hspace{5mm} $\mathcal O_{L_n}$  \hspace{5mm} \\  \hline  \hline 
\rule{0pt}{5mm}   $(-2,3)$ & $67^2 $ & (1) & $67$ & $(1,0)$ & $3\rho +1 $ &   $\Z [G] \alpha$  \\ \hline
\rule{0pt}{5mm}   $(-1,3)$ & $73^2$ & (1) & $73$ & $(1,0)$ & $3\rho$ & $\Z [G] \alpha$ \\ \hline
\rule{0pt}{5mm}   $(1,3)$ & $7^2 \cdot 13^2$ & (1) & $7\cdot 13$ & $(1,0)$ & $3\rho $ &   $\Z [G] \alpha$  \\ \hline
\rule{0pt}{5mm}   $(2,3)$ & $103^2$ & (1) & $103$ & $(1,0)$ & $3\rho-1 $ & $\Z [G] \alpha$ \\ \hline
\end{tabular}
\label{table:3}
\end{center}
%}
\end{table}
 \begin{table}[H]
\caption{$n_2=5$}
%{\small 
 \begin{center}
  \begin{tabular}{|  >{\centering} p{4.5em}|>{\centering} p{4em}|>{\centering} p{3em}|>{\centering} p{4em}|>{\centering} p{3em}|>{\centering} p{6em}| c |}
 \hline
\rule{0pt}{5mm}   $(n_1,n_2)$ & $D_{L_n}$ & Case & $\Delta_n$ & $(a_0,a_1)$ & $\alpha$ & $\mathcal O_{L_n}$ \\  \hline  \hline 
\rule{0pt}{5mm}   $(-4,5)$ & $181^2$ & (1) & $181$ & $(1,0)$ & $5\rho +1 $ & $\Z [G] \alpha$ \\ \hline
\rule{0pt}{5mm}   $(-3,5)$ & $3^4\cdot 7^2$ & (3) (i) & $3^3\cdot 7$ & $(1,-1)$ & $\frac{5}{3} \rho-\frac{5}{3} \rho' $ & $\Z [G] \alpha \oplus \Z$ \\ \hline
\rule{0pt}{5mm}   $(-2,5)$ & $199^2$ & (1) & $199$ & $(1,0)$ & $5\rho +1 $ & $\Z [G] \alpha$ \\ \hline
\rule{0pt}{5mm}   $(-1,5)$ & $211^2 $ & (1) & $211$ & $(1,0)$ & $5\rho $ & $\Z [G] \alpha$ \\ \hline
\rule{0pt}{5mm}   $(1,5)$ & $241^2$ & (1) & $241$ & $(1,0)$ & $5\rho $ & $\Z [G] \alpha$ \\ \hline
\rule{0pt}{5mm}   $(2,5)$ & $7^2\cdot 37^2$ & (1) & $7\cdot 37$ & $(1,0)$ & $5\rho-1 $ & $\Z [G] \alpha$ \\ \hline
\rule{0pt}{5mm}   $(3,5)$ & $3^4 \cdot 31^2$ & (3) (ii) & $3^2\cdot 31$ & $(1,0)$ & $5\rho -1 $ & $\Z [G] \alpha \oplus \Z $ \\ \hline
\rule{0pt}{5mm}   $(4,5)$ & $7^2 \cdot 43^2$ & (1) & $7\cdot 43$ & $(1,0)$ & $5\rho-1 $ & $\Z [G] \alpha$ \\ \hline
\end{tabular}
\label{table:5}
\end{center}
%}
\end{table}
 \begin{table}[H]
\caption{$n_2=7$}
%{\small 
 \begin{center}
  \begin{tabular}{| >{\centering}p{4.5em}|>{\centering} p{4em}|>{\centering} p{3em}|>{\centering} p{4em}|>{\centering} p{3em}|>{\centering} p{6em}| c |}
 \hline
\rule{0pt}{5mm}   $(n_1,n_2)$ & $D_{L_n}$ & Case & $\Delta_n$ & $(a_0,a_1)$ & $\alpha$ & $\mathcal O_{L_n}$ \\  \hline  \hline 
\rule{0pt}{5mm}   $(-6,7)$ & $3^4 \cdot 13^2$ & (1) & $3^3 \cdot 13$ & $(1,-1)$ & $\frac{7}{3} \rho -\frac{7}{3} \rho'$ & $\Z [G] \alpha$ \\ \hline
\rule{0pt}{5mm}   $(-5,7)$ & $19^2$ & (1) & $19^2$ & $(2,5)$ & $ \frac{14}{19} \rho +\frac{35}{19} \rho' +\frac{18}{19} $ & $\Z [G] \alpha$ \\ \hline
\rule{0pt}{5mm}   $(-4,7)$ & $373^2$ & (1) & $373$ & $(1,0)$ & $7\rho +1  $ & $\Z [G] \alpha$ \\ \hline
\rule{0pt}{5mm}   $(-3,7)$ & $3^4 \cdot 43^2$ & (3) (ii)  & $3^2 \cdot 43$ & $(1,0)$ & $7\rho+1 $ & $\Z [G] \alpha \oplus \Z$ \\ \hline
\rule{0pt}{5mm}   $(-2,7)$ & $13^2 \cdot 31^2 $ & (1)& $13\cdot 31$ & $(1,0)$ & $7\rho +1 $ & $\Z [G] \alpha $ \\ \hline
\rule{0pt}{5mm}   $(-1,7)$ & $421^2$ & (1) & $421$ & $(1,0)$ & $7\rho$ & $\Z [G] \alpha $ \\ \hline
\rule{0pt}{5mm}   $(1,7)$ & $463^2$ & (1) & $463$ & $(1,0)$ & $7\rho $ & $\Z [G] \alpha$ \\ \hline
\rule{0pt}{5mm}   $(2,7)$ & $487^2$ & (1) & $487$ & $(1,0)$ & $7\rho-1 $ & $\Z [G] \alpha$ \\ \hline
\rule{0pt}{5mm}   $(3,7)$ & $19^2$ & (2) & $3^3 \cdot 19$ & $(1,-1)$ & $\frac{7}{9} \rho -\frac{7}{9} \rho'+\frac{1}{3} $ & $\Z [G] \alpha$ \\ \hline
\rule{0pt}{5mm}   $(4,7)$ & $541^2$ & (1) & $541$ & $(1,0)$ & $7\rho-1 $ & $\Z [G] \alpha$ \\ \hline
\rule{0pt}{5mm}   $(5,7)$ & $571^2$ & (1)& $571$ & $(1,0)$ & $7\rho -2 $ & $\Z [G] \alpha $ \\ \hline
\rule{0pt}{5mm}   $(6,7)$ & $3^4 \cdot 67^2$ & (3) (ii) & $3^2\cdot 67$ & $(1,0)$ & $7\rho-2 $ & $\Z [G] \alpha \oplus \Z$ \\ \hline
\end{tabular}
\label{table:7}
\end{center}
%}
\end{table}
 \begin{table}[H]
\caption{$n_2=11$}
%{\small 
 \begin{center}
  \begin{tabular}{| >{\centering}p{4.5em}|>{\centering} p{4em}|>{\centering} p{3em}|>{\centering} p{4em}|>{\centering} p{3em}|>{\centering} p{6em}| c |}
 \hline
\rule{0pt}{5mm}   $(n_1,n_2)$ & $D_{L_n}$ & Case & $\Delta_n$ & $(a_0,a_1)$ & $\alpha$ & $\mathcal O_{L_n}$ \\  \hline  \hline 
\rule{0pt}{5mm}   $(-10,11)$ & $859^2$ & (1) & $859$ & $(1,0)$ & $11\rho +3 $ & $\Z [G] \alpha$ \\ \hline
\rule{0pt}{5mm}   $(-9,11)$ & $3^4 \cdot 97^2$ & (3) (ii)  & $3^2 \cdot 97 $ & $(1,0)$ & $11\rho+3 $ & $\Z [G] \alpha \oplus \Z$ \\ \hline
\rule{0pt}{5mm}   $(-8,11)$ & $7^2 \cdot 127^2$ & (1)  & $7\cdot 127 $ & $(1,0)$ & $11\rho +3  $ & $\Z [G] \alpha $ \\ \hline
\rule{0pt}{5mm}   $(-7,11)$ & $907^2$ & (1) & $907$ & $(1,0)$ & $11\rho+2 $ & $\Z [G]\alpha$ \\ \hline
\rule{0pt}{5mm}   $(-6,11)$ & $3^4\cdot 103^2$ & (3) (ii) & $3^2 \cdot 103 $ & $(1,0)$ & $11\rho +2 $ & $\Z [G] \alpha \oplus \Z$ \\ \hline
\rule{0pt}{5mm}   $(-5,11)$ & $13^2 \cdot 73^2 $ & (1) & $13\cdot 73 $ & $(1,0)$ & $11\rho +2 $ & $\Z [G] \alpha $ \\ \hline
\rule{0pt}{5mm}   $(-4,11)$ & $7^2 \cdot 139^2$ & (1) & $7\cdot 139 $ & $(1,0)$ & $11\rho +1 $ & $\Z [G] \alpha$ \\ \hline
\rule{0pt}{5mm}   $(-3,11)$ & $37^2$ & (2) & $3^3 \cdot 37 $ & $(1,-1)$ & $\frac{11}{9} \rho -\frac{11}{9} \rho'-\frac{1}{3} $ & $\Z [G] \alpha$ \\ \hline
\rule{0pt}{5mm}   $(-2,11)$ & $13^2 \cdot 79^2$ & (1)& $13 \cdot 79$ & $(1,0)$ & $11\rho+1$ & $\Z [G] \alpha $ \\ \hline
\rule{0pt}{5mm}   $(-1,11)$ & $7^2 \cdot 151^2$ & (1) & $7 \cdot 151$ & $(1,0)$ & $11\rho $ & $\Z [G] \alpha$ \\ \hline
\rule{0pt}{5mm}   $(1,11)$ & $1123^2$ & (1) & $1123$ & $(1,0)$ & $11\rho $ & $\Z [G] \alpha$ \\ \hline
\rule{0pt}{5mm}   $(2,11)$ & $19^2 \cdot 61^2$ & (1) & $19\cdot 61$ & $(1,0)$ & $11\rho-1 $ & $\Z [G] \alpha$ \\ \hline
\rule{0pt}{5mm}   $(3,11)$ & $3^4 \cdot 7^2 \cdot 19^2$ & (3) (ii) & $3^2 \cdot 7 \cdot 19$ & $(1,0)$ & $11\rho -1  $ & $\Z [G] \alpha \oplus \Z$ \\ \hline
\rule{0pt}{5mm}   $(4,11)$ & $1237^2$ & (1) & $1237$ & $(1,0)$ & $11\rho-1 $ & $\Z [G] \alpha$ \\ \hline
\rule{0pt}{5mm}   $(5,11)$ & $1279^2$ & (1)& $1279$ & $(1,0)$ & $11\rho -2 $ & $\Z [G] \alpha $ \\ \hline
\rule{0pt}{5mm}   $(6,11)$ & $3^4 \cdot 7^2$ & (3) (i) & $3^3 \cdot 7^2$ & $(5,1)$ & $\frac{55}{21} \rho+\frac{11}{21}\rho'-\frac{4}{7}$ & $\Z [G] \alpha \oplus \Z$ \\ \hline
\rule{0pt}{5mm}   $(7,11)$ & $37^2$ & (1) & $37^2$ & $(7,3)$ & $\frac{77}{37}\rho+\frac{33}{37}\rho'-\frac{11}{37} $ & $\Z [G] \alpha$ \\ \hline
\rule{0pt}{5mm}   $(8,11)$ & $13^2 \cdot 109^2$ & (1) & $13 \cdot 109 $ & $(1,0)$ & $11\rho-3$ & $\Z [G] \alpha$ \\ \hline
\rule{0pt}{5mm}   $(9,11)$ & $3^4 \cdot 163^2$ & (3) (ii) & $3^2 \cdot 163$ & $(1,0)$ & $11\rho-3 $ & $\Z [G] \alpha \oplus \Z$ \\ \hline
\rule{0pt}{5mm}   $(10,11)$ & $7^2 \cdot 31^2$ & (1) & $7^2 \cdot 31$ & $(3,1)$ & $\frac{33}{7} \rho+\frac{11}{7} \rho'-\frac{11}{7}$ & $\Z [G] \alpha$ \\ \hline
\end{tabular}
\label{table:11}
\end{center}
%}
\end{table}
 
 %%%%%%%%%%%%%%%%%%%%%%%%%%%%%%%%%%%%%%%%%%%%%%%%


\begin{thebibliography}{10}
%
%
\bibitem{AF}
V.~Acciaro and C.~Fieker, Finding normal integral bases of cyclic number fields of prime degree, J. Symbolic Comput. 30, no.2, 129--136 (2000).
%
\bibitem{Al}
A.~A.~Albert, A determination of the integers of all cubic fields, Ann. of Math., (2) 31, 550--566 (1930).
%
%\bibitem{Ar}
%V.~A.~Artamonov and A. A. Bodvi, Integral group rings : groups of
%invertible elements and classical K-theory, (Russian) Translated in J.
%Soviet Math. 57, no.2, 2931--2958 (1991).
%
%\bibitem{Ch}
%A. Ch\^{a}telet, Arithm\'{e}tique des corps ab\'{e}lians du troisi\`{e}me  degr\'{e}, Ann. Sci. \'{E}cole Norm. sup. (4) 63, 109--160 (1946).
%
\bibitem{A1} M.~Aoki, Gaussian periods and Shanks' cubic polynomials. I,  to appear in the proceedings of the conference ICCGNFRT-2024 (2024).
%
\bibitem{A2} M.~Aoki, Gaussian periods and Shanks' cubic polynomials. II, submitted (2025).
%
\bibitem{BCP} W.~Bosma, J.~Cannon  and C.~Playoust, The Magma algebra system. I. The user language, J. Symbolic Comput., 24, 235–265 (1997).
%
\bibitem{C}
T.~W.~Cusick, Lower bounds for regulators, in Number Theory,
Lecture Notes in Math.~1068,
Springer, Berlin,  63--73 (1984). 
%
\bibitem{G}
C.~F.~Gauss,
Disquisitiones arithmeticae (Translated by Arthur A. Clarke),
New Haven-London: Yale University Press (1966).
%
%\bibitem{D}
%B.~N.~Delon and D. K. Faddeev, The theory of Irrationalities of the third
%degree, Translations of Mathematical Monographs 10, American Mathematical
%Society, Providence, R.I. (1964). 
%
%\bibitem{E}
%G.~R.~Everest, Counting generators of normal integral bases, Amer. J.
%Math. 120, no.5, 1007--1018 (1998).
%
%\bibitem{G}
%M. N. Gras, Sur les corps cubiques cycliques dont l'anneau des entiers est monog\`ene, Ann. Sci. Univ. Besan\c{c}on Math. (3), no. 6 (1973).
%
\bibitem{HA1} Y.~Hashimoto and M.~Aoki, Normal integral bases and Gaussian periods in the simplest cubic fields, 
Ann.~Math.~du Qu\'{e}bec 48, 157--173 (2024).
%
\bibitem{HA2} Y.~Hashimoto and M.~Aoki, Normal integral bases of Lehmer's cyclic quintic fields, Ramanujan J.~65, no.2, 985--1010 (2024).
%
%\bibitem{HW}
%S.~A.~Hambleton and H. C. Williams, Cubic fields with geometry,
%CMS Books in Mathematics, Springer, Cham, 2018.
%
\bibitem{H} H.~Hasse, 
Arithmetische Bestimmung von Grundeinheit und Klassenzahl in zyklischen kubischen und biquadratischen Zahlk\"{o}rpern,
Abh. Deutsch. Akad. Wiss. Berlin, Math.-Naturw. Kl., No. 2, 95 pp. (1948).
%
\bibitem{KS}
T. Kashio and R. Sekigawa, The characterization of cyclic cubic fields with power integral bases,
Kodai Math. J. ~44, no.~2, 290--306 (2021).
%
\bibitem{K}
T.~Komatsu, Cyclic cubic field with explicit Artin symbols,
Tokyo J. Math.~30 , no. 1, 169–-178 (2007).
%
%
%\bibitem{La}
%A. J. Lazarus, Gaussian periods and units in certain cyclic fields, Proc. Amer. Math. Soc. 115, no.~4, 961--968 (1992).
%
\bibitem{Le}
G.~Lettl, The ring of integers of an abelian number field, J.~reine.~angew.~Math.~404, 162--170 (1990).
%
%\bibitem{L}
%R. Long, Algebraic number theory, Marcel Dekker, New York (1977).
%
%\bibitem{Le}
%E. Lehmer, Connection between Gaussian periods and cyclic units. 
%Math. Comp. 50, no.182, 535--541 (1988).
%
\bibitem{Leo}
H.~W.~Leopoldt, \"{U}ber die Hauptordnung der ganzen Elemente eines abelschen Zahlk\"{o}rpers, J.~
reine angew.~Math. 201, 119--149 (1959).  
%
%\bibitem{N} 
%W.~Narkiewicz, Elementary and Analytic Theory of Algebraic 
%Numbers, Springer-Verlag Berlin Heidelberg (2004).
%
\bibitem{OA} 
H.~Ogawa and M.~Aoki, Galois module structure of algebraic integers of the simplest cubic field, submitted.
%
\bibitem{Se} 
J.~P,~Serre, Topics in Galois theory,  2nd ed.,  with notes by Henri Darmon, Research Notes in Mathematics, vol. 1, A K Peters, Ltd., Wellesley, MA, 2008.
%
\bibitem{S} 
D.~Shanks, The simplest cubic fields, Math. Comp. 28, 1137--1152 (1974).
%
%\bibitem{V}
%G.~Voronoi, Concerning algebraic integers derivable from a root of an
%equation of the third degree, Master's Thesis, St. Petersburg (1894). 
%
\bibitem{W1}
L.~C.~Washington, Introduction to Cyclotomic Fields, 2nd ed.,  GTM 83, Springer-Verlag, New York, 1997.
%
\bibitem{W2} 
L.~C.~Washington, Class numbers of the simplest cubic 
fields. Math. Comp. 48, no.177, 371--384 (1987).
%
\end{thebibliography}
\end{document}